\documentclass[11pt,reqno]{amsart}
\usepackage[a4paper,margin=1in]{geometry}
\usepackage{amsmath,amssymb,amsthm,mathtools}
\usepackage[hidelinks]{hyperref}
\usepackage{enumitem}
\usepackage{needspace}

\newtheorem{theorem}{Theorem}[section]
\newtheorem{proposition}[theorem]{Proposition}
\newtheorem{lemma}[theorem]{Lemma}
\newtheorem{corollary}[theorem]{Corollary}

\theoremstyle{definition}
\theoremstyle{remark}

\newtheorem{conjecture}[theorem]{Conjecture}

\newcommand{\Spec}{\operatorname{Spec}}

\newcommand{\Vol}{\operatorname{Vol}}
\newcommand{\C}{\mathbb C}

\makeatletter
\newcommand{\subsectionnotoc}{\@startsection{subsection}{\@M}%
  \z@{.5\linespacing\@plus.7\linespacing}{-.5em}%
  {\normalfont\bfseries}}
\makeatother

\title{Some uniform estimates for K\"ahler--Ricci Shrinkers}
\author{Junsheng Zhang}
\address{Courant Institute of Mathematical Sciences\\
  New York University, 251 Mercer St\\
  New York, NY 10012\\}
\email{jz7561@nyu.edu}
\hypersetup{
  pdftitle={Uniform estimates for smooth K\"ahler--Ricci Shrinkers},
  pdfauthor={Junsheng Zhang}
}
\date{}
\begin{document}

\begin{abstract}
 In this paper, we establish several uniform estimates for K\"ahler--Ricci
shrinkers without imposing any curvature assumptions. In particular, we
prove:
\begin{enumerate}
    \item a uniform lower bound for the entropy;
    \item a uniform lower bound for the asymptotic volume ratio of
    K\"ahler--shrinkers with maximal volume growth;
    \item a uniform lower bound for the scalar curvature on balls centered at a
minimum point of the soliton potential for non-Gaussian shrinkers.
\end{enumerate}
\end{abstract}

\maketitle

\tableofcontents
\section*{Introduction}
By a K\"ahler--Ricci shrinker, we mean a complete shrinking gradient K\"ahler--Ricci soliton, which arises naturally as a model for finite-time singularities of the K\"ahler--Ricci flow. It is proved in \cite{SZ} that every K\"ahler--Ricci shrinker admits a polarized Fano fibration structure, which plays the role of a noncompact Fano manifold.  This
algebraic structure and the techniques behind its construction, opens the door to 
uniform estimates for K\"ahler--Ricci shrinkers and, ultimately, to their
moduli theory.

Throughout the paper, a smooth K\"ahler--Ricci shrinker is
normalized by
\begin{equation}\label{eq:canonical-shrinker-normalization}
 \operatorname{Ric}(\omega)+\sqrt{-1}\partial\bar\partial f=\omega,
 \qquad
 \Delta_\omega f+f-|\nabla^{1,0}f|^2=0.
\end{equation}We choose
\(p\in\operatorname*{argmin}f\), and put \(\rho^2=f+n\).
The holomorphic Killing vector field \(\xi=J\nabla f\) is called the
\emph{soliton vector field}. Following the notation of \cite{SZ}, we define
the \textit{weighted volume} by
\begin{equation}\label{eq:weighted-volume-definition}
    \mathbb W(\xi)
    :=\frac{1}{(2\pi)^n}
    \int_X e^{-f}\frac{\omega^n}{n!}
    =e^{\boldsymbol\mu(g)+n},
    \qquad
    \boldsymbol\mu(g)=\mu\bigl(g,\tfrac12\bigr),
\end{equation}
where $\mu$ denotes Perelman's $\mu$-entropy.

\begin{theorem}
\label{int:uniform-entropy}
Let \((X^n,\xi)\) be a smooth K\"ahler--Ricci
shrinker.  Then
\[
 \mathbb W(\xi)\geq\frac{1}{N_n^n n!},
\]
where \(
 N_n:=2^{n+1}(n+1)!\,n(n+1).\) Equivalently,
 $\boldsymbol\mu(g)\geq-n-n\log N_n-\log(n!).$
\end{theorem}

The uniform lower bound for the weighted volume in the compact case of
Theorem~\ref{int:uniform-entropy} was proved in
\cite{guo2022compactness}. Our argument treats the compact and noncompact
cases simultaneously and yields an explicit lower bound. A crucial ingredient in the proof is the algebraic
characterization of \(\mathbb W(\xi)\) established in \cite{SZ}.

\

The resulting entropy bound provides the hypothesis needed to apply
\cite{LLW}. Thus, in each fixed dimension, every sequence of smooth
K\"ahler--Ricci shrinkers centered at minima of their soliton potentials
admits a subsequence converging in the pointed
\(\widehat C^\infty\)-Cheeger--Gromov sense,
\begin{equation}\label{metric-convergence}
   (X_i,\omega_i,J_i,f_i,p_i)
 \longrightarrow
 (X_\infty,d_\infty,\omega_\infty,J_\infty,f_\infty,p_\infty),
\end{equation}
and the singular set of the limit has Minkowski codimension at least four
\cite[Theorem~1.1, Remark~7.3 and (7.15)]{LLW}. For each smooth term, let
\[
 \pi_i:X_i\to Y_i
\]
denote its Fano fibration. 
The singular-fibration theorem
\cite{HZ} shows that the
limit \(X_\infty\) also carries a polarized Fano fibration structure
\[
 \pi_\infty:X_\infty\longrightarrow Y_\infty.
\] 

For a smooth
K\"ahler--Ricci shrinker, it was proved in
\cite[Proposition~7.2]{LZ} that its volume-growth order is equal to the real
dimension of the affine cone associated with its polarized Fano fibration.
The same relation extends to the singular limit \(X_\infty\); see
Theorem~\ref{thm:singular-high-level-dh}.

In the next result, we show that the dimension of the associated affine cone,
equivalently the volume-growth order, is preserved under taking limits.
This has useful consequences for the moduli spaces of smooth
K\"ahler--Ricci shrinkers. Let \(\mathcal{KRS}(n)\) denote the space of
\(n\)-dimensional smooth K\"ahler--Ricci shrinkers, and let
\(\mathcal{KRS}(n,k)\subset \mathcal{KRS}(n)\) denote the subspace consisting
of those with volume-growth order \(2k\), where
\(k\in\{0,1,\ldots,n\}\). Then by \cite{SZ} and \cite[Proposition 7.2]{LZ}, we know that $\mathcal{KRS}(n,0)$ consists precisely of compact K\"ahler--Ricci shrinkers and 
\[
 \mathcal{KRS}(n)
    =
    \bigsqcup_{k=0}^n \mathcal{KRS}(n,k).
\] For any such space, we use an overline to denote
its completion with respect to pointed
\(\widehat C^\infty\)-Cheeger--Gromov convergence.

\begin{theorem}
\label{int:affine-base-dimension-constancy}
For any convergent sequence as in \eqref{metric-convergence}, one has for all sufficiently large \(i\),
\[
    \dim_{\mathbb C} Y_i=\dim_{\mathbb C} Y_\infty.
\]
Consequently, the completion of the moduli space decomposes as the disjoint
union
\begin{equation}
    \overline{\mathcal{KRS}(n)}
    =
    \bigsqcup_{k=0}^n \overline{\mathcal{KRS}(n,k)}.
\end{equation}
\end{theorem}

\

As a consequence of the above result and the uniform integral curvature
estimate of \cite[Corollary~6.24]{LW}, we obtain the following uniform lower
bound for the volume ratio of K\"ahler--Ricci shrinkers with maximal volume
growth. We note that the corresponding uniform upper bound was established in the general Riemannian setting in
\cite{CaoZhou,HM}.

\begin{theorem} 
\label{int:uniform-avr-gap}
For every \(n\geq1\), there exists a constant \(\delta_n>0\) such that every
complex \(n\)-dimensional K\"ahler--Ricci shrinker with maximal volume growth,
i.e. every element of \(\mathcal{KRS}(n,n)\), satisfies
\[
    \lim_{r\to\infty}
    \frac{\operatorname{Vol}_g B_g(p,r)}{r^{2n}}
    \geq \delta_n.
\]
\end{theorem}

\

Using the noncompact
Duistermaat--Heckman localization \cite{PratoWu}, applied directly to the
possibly irrational soliton field, we obtain the polynomial volume growth of every K\"ahler--Ricci shrinker.  Then using
 shrinker identities, we obtain the precise growth order for the integral of scalar curvature.

\begin{theorem}
\label{int:polynomial-volume-scalar-mass}
Let \(X\in \mathcal{KRS}(n,k)\), where \(k\in\{1,\ldots,n\}\). Then the
following hold.
\begin{enumerate}
    \item For all sufficiently large \(s\), the function
    \[
        s\longmapsto \operatorname{Vol}\{\rho^2<s\}
    \]
    agrees with a polynomial of degree \(k\) with positive leading
    coefficient.

    \item One has
    \[
        0<
        \lim_{r\to\infty}
        \frac{\operatorname{Vol}_g B_g(p,r)}{r^{2k}}
        <\infty,
        \qquad
        \lim_{r\to\infty}
        \frac{\int_{B_g(p,r)} R_\omega\,\omega^n/n!}
             {\operatorname{Vol}_g B_g(p,r)}
        =n-k.
    \]

    \item If \(k=n\), then the limit
    \[
        \lim_{r\to\infty}
        r^{2-2n}
        \int_{B_g(p,r)} R_\omega\,\frac{\omega^n}{n!}
    \]
    exists and is positive unless \(X\) is Gaussian.
\end{enumerate}
\end{theorem}

Motivated by Theorem \ref{int:uniform-avr-gap}, it is natural to conjecture that the asymptotic volume ratio for elements in $\mathcal{KRS}(n,k)$ also has a uniform bound depending only on $n$ and $k$; see Conjecture \ref{conj:uniform-volume-growth}. We prove the case $k=1$ in the following and the intermediate cases remain open.
\begin{theorem}
\label{int:rank-one-volume-coefficient}
For every \(n\geq1\), there exist constants
\(
    0< C_{n,1}<\infty
\)
such that every \(X\in\mathcal{KRS}(n,1)\) satisfy
\[
    \frac{1}{C_{n,1}}
    \leq
    \lim_{r\to\infty}
    \frac{\operatorname{Vol}_g B_g(p,r)}{r^2}
    \leq
    C_{n,1}.
\]
\end{theorem}

\

As a noncompact analogue of \cite{DS}, it is conjectured in
Conjecture~\ref{conj:flat-convergence} that the convergence in
\eqref{metric-convergence} can be realized within a flat family. Motivated
by this conjecture, we prove the following cone-isolation result.

\begin{theorem}[Cone isolation]\label{int:cone-isolation}
The boundary
\(
    \overline{\mathcal{KRS}(n)}\setminus\mathcal{KRS}(n)
\)
contains no (singular) Ricci-flat K\"ahler cones. More precisely, for a
convergent sequence as in \eqref{metric-convergence}, if the limit
\(X_\infty\) is Ricci-flat on its regular locus, then \(X_i\) is the Gaussian
shrinker for all sufficiently large \(i\).
\end{theorem}

Following the argument of \cite{LLW,LW}, we obtain the following
scalar-curvature gap.

\begin{theorem}
\label{int:scalar-curvature-gap}
For every \(n\geq1\), there exists a constant \(\epsilon_n>0\) such that
every complete smooth non-Gaussian \(n\)-dimensional K\"ahler--Ricci
shrinker, satisfies
\[
    \inf_{B_g(p,1)} R_\omega \geq \epsilon_n,
\]
where \(p\) is any minimum point of the soliton potential.
\end{theorem}

\    

Independently of the above results, by further developing the method of
\cite[Section~4]{LZ}, we prove the following scalar curvature decay criterion for
asymptotically conical K\"ahler--Ricci shrinkers. 

\begin{theorem}
\label{int--thm:quadratic-scalar-decay-implies-ac}
Let \((X^n,\omega,f)\) be a noncompact K\"ahler--Ricci shrinker.
Assume
\[
        R_\omega(x)\rightarrow 0,\quad  \text{as $x\rightarrow \infty$. }
\] 
Then \((X,\omega)\) is asymptotically conical.
\end{theorem}

The paper is organized as follows.
Section~\ref{sec:uniform-entropy} proves the uniform weighted-volume bound,
deduces weak compactness in fixed dimension, and establishes constancy of
the affine-base dimension under convergence.  Section~\ref{sec:effective-avr-gap}
studies volume growth and integral of scalar-curvature.  It proves the uniform
asymptotic-volume-ratio gap in the maximal volume growth case, the precise
scalar-curvature-mass asymptotics, and the uniform two-sided volume-ratio
bound when the affine base is one-dimensional.
Section~\ref{sec:gaussian-scalar-lifting} proves weighted integrability for pluri-anticanonical sections, followed by a
global H\"ormander lifting construction.  Section~\ref{sec:limit-cone-setup}
uses the results proved in Section \ref{sec:gaussian-scalar-lifting} to prove the cone isolation result. Then as an application, we derive the scalar-curvature gap.
Section~\ref{sec:scalar-decay-ac}
derives estimates comparing the shrinker metric with the ambient background
metric under a bounded scalar-curvature assumption and proves that
scalar-curvature decay implies asymptotic conicality.  Finally,
Section~\ref{sec:filtered-partial-c0-flat} discusses two conjectures and proves a
uniform dimension estimate for pluri-anticanonical sections with finite degree.

\subsectionnotoc{Acknowledgements}
The author thanks Yu Li for helpful discussions and useful suggestions that
improved the paper, and Hanbing Fang for helpful discussions at an early
stage of this work.
\subsectionnotoc{Declaration on the use of AI.}
The main ideas developed in this paper are the author's own. AI tools
(ChatGPT Pro with GPT-5.6) were used to assist in the preparation of the
manuscript. The author takes full responsibility for the paper's content
and correctness.

\section{Uniform entropy and weak compactness}
\label{sec:uniform-entropy}
Using the algebraic formulation of the weighted volume in \cite{SZ},
together with the uniform relative very ampleness \cite{Fujino}, we obtain a
dimension-only lower bound for the weighted volume. The compactness theorem
of Li--Li--Wang \cite{LLW} then yields weak precompactness. We next use the
lifting of holomorphic functions from the limit spaces to establish the
constancy of the dimension of the affine base. 

Recall that a Fano
fibration is a surjective projective morphism \(\pi\colon X\to Y\) of
normal complex algebraic varieties, with \(X\) klt,
\(\pi_*\mathcal O_X=\mathcal O_Y\), and \(-K_X\) \(\pi\)-ample. For a polarized Fano fibration $(\pi\colon X\to Y,\xi)$, an algebraic
weighted volume $\mathbb W(\xi)$ is defined in \cite[Section~5]{SZ}.
When $\xi$ is the soliton vector field of a K\"ahler--Ricci shrinker,
it agrees with the analytic weighted volume defined in
\eqref{eq:weighted-volume-definition}; see
\cite[Proposition~5.9]{SZ}. Therefore the following result implies Theorem \ref{int:uniform-entropy}.

\begin{theorem}[Weighted-volume bound]
\label{thm:uniform-entropy}
Let \((\pi:X^n\rightarrow Y,\xi)\) be a  smooth polarized Fano fibration.  Then
\begin{equation}\label{eq:uniform-weighted-volume}
 \mathbb W(\xi)\geq\frac{1}{N_n^n n!}.
\end{equation}
\end{theorem}

\begin{proof}
Here \(Y=\Spec R\) is an polarized affine cone; in the compact case it is a
point.  Let \(\mathbb T\) be the closure of the one-parameter group
generated by \(\xi=J\nabla f\).  The canonical \(\mathbb T\)-action on
\(-K_X\) gives weight decompositions
\[
 H^0(X,-kK_X)=\bigoplus_\alpha R_{k,\alpha}.
\]
Throughout this paper, \(H^0\) denotes algebraic global sections unless otherwise specified. The
characters \(\alpha\) lie in the character lattice of \(\mathbb T\),
viewed inside \(\operatorname{Lie}(\mathbb T)^*\), and the vector spaces \(R_{k,\alpha}\) are
finite-dimensional. 
By \cite[Section 4 and Section 5]{SZ}, we have
\begin{equation}\label{eq:character-formula}
 \mathbb W(\xi)=
 \lim_{k\to\infty}\frac{1}{k^n}
 \sum_\alpha \dim R_{k,\alpha}
 \exp\!\left(-\frac{\langle\alpha,\xi\rangle}{k}\right).
\end{equation}

The algebraic input is a uniform relative very-ampleness from \cite[Corollary~1.4, its proof, and Remark~1.3]{Fujino}:
\begin{equation}\label{eq:uniform-relative-embedding}
        \text{\(-N_nK_X\) is
\(\pi\)-very ample. }
\end{equation}

Let \(q\) be a minimum point of \(f\).  Such a point exists since  \(f\) is proper and bounded below
\cite[Theorem~1.1]{CaoZhou}.
 Since \(Y\) is affine and
\(-N_nK_X\) is \(\pi\)-very ample, relative generation and equivariance
give a \(\mathbb T\)-eigensection
\(s_0\in H^0(X,-N_nK_X)\) with \(s_0(q)\ne0\).
Taking weight components of finitely many sections defining the local
embedding and adjoining \(hs_0\) for finitely many homogeneous \(h\in R\), we obtain eigensections
\(s_1,\ldots,s_n\) such that
\begin{equation}\label{eq:eigen-coordinates}
 z_j:=\frac{s_j}{s_0},\qquad 1\leq j\leq n,
\end{equation}
form a regular system of parameters at \(q\).  Write $\operatorname{wt}_\xi(s_i)$ for the weight of $s_i$ under the $\xi$-action, that is $\mathcal{L}_{\xi}s_i=\operatorname{wt}_\xi(s_i)$.

Write
\begin{equation}\label{eq:coordinate-weights}
 \mathcal L_\xi z_j=\sqrt{-1}\lambda_jz_j.
\end{equation}
The Hamiltonian normal form at
the minimum \(q\) gives \(\lambda_j\geq0\). Moreover as a consequence of \eqref{eq:coordinate-weights}, we have 
\begin{equation}\label{canonical coordinate}
    \mathcal L_\xi
 \left(\frac{\partial}{\partial z_1}\wedge\cdots\wedge
 \frac{\partial}{\partial z_n}\right)
 =-\sqrt{-1}\left(\sum_{i=1}^n\lambda_i\right)
 \left(\frac{\partial}{\partial z_1}\wedge\cdots\wedge
 \frac{\partial}{\partial z_n}\right).
\end{equation}

The canonical lift used in the character formula satisfies
\begin{equation}\label{eq:canonical-lift}
 \mathcal L_\xi\sigma=\sqrt{-1}f(q)\sigma,
 \qquad \sigma\in(-K_X)_q,
\end{equation}
by \cite[Equation~(5.18)]{SZ}.
Comparing \eqref{canonical coordinate} with \eqref{eq:canonical-lift} gives
\begin{equation}\label{eq:weight-budget}
 \lambda_j\geq0,\qquad \sum_{j=1}^n\lambda_j=-f(q),
 \qquad f(q)\leq0.
\end{equation}
Since evaluation at \(q\) is equivariant, \(s_0(q)\neq0\), and
\(\mathcal L=-N_nK_X\), we have
\begin{equation}\label{eq:s0-weight}
 \operatorname{wt}_\xi(s_0)=N_nf(q).
\end{equation}
Since \(s_i=z_is_0\), it follows that
\begin{equation}\label{eq:si-weight}
 \operatorname{wt}_\xi(s_i)=N_nf(q)+\lambda_i.
\end{equation}

For a multi-index \(I=(I_1,\ldots,I_n)\in\mathbb Z_{\geq0}^n\) with
\(|I|\leq m\), define
\begin{equation}\label{eq:monomial-sections}
 \sigma_I:=s_0^{m-|I|}\prod_{i=1}^n s_i^{I_i}
 =s_0^m z^I\in H^0(X,-mN_nK_X).
\end{equation}
The \(\binom{m+n}{n}\) sections are linearly independent: after restricting
near \(q\) and dividing by \(s_0^m\), a relation would be a polynomial relation
among the local coordinates \(z_1,\ldots,z_n\).   By
\eqref{eq:weight-budget}--\eqref{eq:si-weight}, their normalized
weights satisfy
\begin{align}\label{eq:monomial-weight-bound}
 \frac{\operatorname{wt}_\xi(\sigma_I)}{N_nm}
 &=f(q)+\frac{\sum_iI_i\lambda_i}{N_nm}\notag\\
 &\leq f(q)+\frac{\sum_i\lambda_i}{N_n}
 =\left(1-\frac{1}{N_n}\right)f(q)\leq0.
\end{align}
 Thus the
character term in \eqref{eq:character-formula} contributed by such monomials is at least one.
Let \(C_k\) denote the normalized character on the right-hand side of
\eqref{eq:character-formula}.  Positivity of every character term, linear
independence of the \(\sigma_I\), and
\eqref{eq:monomial-weight-bound} give
\begin{align*}
 C_{N_nm}\geq \frac{1}{(N_nm)^n}
   \sum_{|I|\leq m}
   \exp\!\left(-\frac{\operatorname{wt}_\xi(\sigma_I)}{N_nm}\right)\geq \frac{\binom{m+n}{n}}{(N_nm)^n}.
\end{align*}
Letting \(m\to\infty\) gives
\eqref{eq:uniform-weighted-volume}:
\[
 \mathbb W(\xi)=\lim_{m\to\infty}C_{N_nm}
 \geq\frac{1}{N_n^n n!}.
\]
\end{proof}

Theorem~\ref{thm:uniform-entropy} supplies the entropy hypothesis in
\cite[Theorem 1.1, Remark~7.3 and (7.15)]{LLW}, which gives the following
compactness statement.
\begin{corollary}
\label{cor:weak-compactness}
Fix \(n\geq2\). Every sequence of complete smooth
K\"ahler--Ricci shrinkers of complex dimension \(n\), centered at
minimum points of their soliton potentials, has a subsequence
converging in the pointed
\(\widehat C^\infty\)-Cheeger--Gromov sense  to a limit in the sense of \cite[Theorem~1.1]{LLW}. Its
singular set has Minkowski codimension at least four.
\end{corollary}

 Let
\begin{equation}\label{eq:general-LLW-convergence}
 (X_i,\omega_i,J_i,f_i,p_i)
 \longrightarrow
 (X_\infty,d_\infty,\omega_\infty,J_\infty,f_\infty,p_\infty).
\end{equation}
be pointed \(\widehat C^\infty\)-Cheeger--Gromov convergence in
the sense of \cite[Theorem~1.1 and Remark~7.3]{LLW}.  Thus the whole
spaces converge in the pointed Gromov--Hausdorff sense, while
\((\omega_i,J_i,f_i)\) converge smoothly through compatible embeddings
on compact subsets of
\(\mathcal R\), the regular part of $X_\infty$.  The function \(f_\infty\)
is proper, its sublevel sets are compact, and
\(\mathcal S:=X_\infty\setminus\mathcal R\) has Minkowski
codimension at least four.

The result in \cite{SZ} gives a natural polarized Fano fibration on every
smooth term, and the singular-fibration theorem \cite{HZ} gives the limit fibration.
We use the notation
\begin{equation}\label{eq:general-natural-fibrations}
 \pi_i:X_i\longrightarrow Y_i,
 \qquad
 \pi_\infty:X_\infty\longrightarrow Y_\infty,
\end{equation}to denote the polarized Fano fibrations.

\begin{lemma}[Lifting limit holomorphic functions]
\label{cor:ab-base-function-lifting}
For every regular function  \(h\) on $X_{\infty}$,  there exists regular function
\(h_{i}\) on $X_i$ that converge to \(h\) uniformly on every
fixed pointed metric ball and smoothly on compact subsets of
\(\mathcal R\).
\end{lemma}

We defer its proof to Section~\ref{sec:gaussian-scalar-lifting} and
now prove dimension constancy of the affine base uisng this result.

\begin{theorem}
\label{thm:affine-base-dimension-constancy}
Let \(Y_i\) and \(Y_\infty\), as in
\eqref{eq:general-natural-fibrations}, denote the affine bases of the Fano
fibrations associated with the converging sequence in
\eqref{eq:general-LLW-convergence}. Then, for all sufficiently large \(i\),
\[
    \dim_{\mathbb C} Y_i=\dim_{\mathbb C} Y_\infty.
\]
\end{theorem}

\begin{proof}
Put \(d=\dim_{\mathbb C}Y_\infty\).  Suppose first that \(d>0\), then 
  apply
Lemma~\ref{cor:ab-base-function-lifting} to the components of $\pi_{\infty}$ and
write \(H_i:X_i\to\mathbb C^N\) for the resulting global holomorphic functions on $X_i$.

Note that there exists a point
\(x_\infty\in\mathcal R\) at which
\(\operatorname{rank}d\pi_\infty=d\).  Smooth convergence on the regular-locus 
then gives \(\operatorname{rank}dH_i=d\) near the corresponding point
for all large \(i\).  Since \(H_i\) factors through \(\pi_i\), this proves
\begin{equation}\label{eq:affine-base-lower-dimension}
 d\leq\dim_{\mathbb C}Y_i.
\end{equation}

For the reverse inequality, pick a point \(x_\infty\in\mathcal R\) with 
\(\operatorname{rank}dH_\infty=d\) and put
\(q=\pi_\infty(x_\infty)\).  Properness of \(\pi_\infty\) makes the 
level sets \(\pi_\infty^{-1}(q)\) compact.   Pointed Gromov--Hausdorff convergence of the space and the convergence of functions transfers this
 to the compactness of the level sets of \(H_i\).
We then choose $x_i$ in both the full-rank locus of \(H_i\) and the full-rank locus of $\pi_i$ such that \(x_i\to x_\infty\).  Put \(q_i=H_i(x_i)\); then \(q_i\to q\).  Let \(C_i\) be a connected
component of the global level \(H_i^{-1}(q_i)\) through \(x_i\).  By the property of Fano fibrations, we have
\(C_i\subset\pi_i^{-1}(\pi_i(x_i))\).  Since
\(\operatorname{rank}dH_i(x_i)=d\), $C_i$ is smooth of local
dimension \(n-d\) at \(x_i\).  The point \(x_i\) lies in the
generic-fiber locus of \(\pi_i\), whose local fiber dimension is
\(n-\dim_{\mathbb C}Y_i\).  Comparing the two germs gives
\[
 n-d\leq n-\dim_{\mathbb C}Y_i,
 \qquad\text{hence}\qquad
 \dim_{\mathbb C}Y_i\leq d.
\]

The case \(d=0\) is immediate. Indeed, \(X_\infty\) is then compact, which
implies that \(X_i\) is compact for all sufficiently large \(i\). Hence
\(
    \dim_{\mathbb C} Y_i=0
\)
for all sufficiently large \(i\).
\end{proof}

\begin{proof}[Proof of Theorem~\ref{int:affine-base-dimension-constancy}]
The dimension equality is Theorem~\ref{thm:affine-base-dimension-constancy}. Combining this with \cite{SZ,LZ}, we obtain the desired disjoint strata of  $\overline{\mathcal{KRS}(n)}$.
\end{proof}

Another application of Theorem \ref{thm:uniform-entropy} is the uniform control of the
local volume ratio on limits of smooth K\"ahler--Ricci shrinkers.  For
\(X\in\overline{\mathcal{KRS}(n)}\) and \(q\in X\), set
\begin{equation}
 \label{eq:int-pointwise-volume-density}
 \Theta_{X,q}:=\lim_{s\downarrow0}
 \frac{\operatorname{Vol}_g(B_d(q,s))}
      {\omega_{2n}s^{2n}}.
\end{equation}

\begin{theorem}
For  every
\(X\in\overline{\mathcal{KRS}(n)}\), and every \(q\in X\),
\begin{equation}
 \label{eq:int-uniform-boundary-volume-density}
 1\geq\Theta_{X,q}
 \geq\frac{e^{-n}}{N_n^n n!}.
\end{equation}
\end{theorem}

\begin{proof}
Let \(X_i\) be smooth K\"ahle--Ricci shrinkers converging to \((X, \omega, f)\) as in \eqref{eq:general-LLW-convergence}. Then global weighted-measure continuity
\cite[Theorem~1.2(c)]{LLW} and uniform estimate on the soliton potential and volume growth of shrinkers, give
\[
 W:=\frac{1}{(2\pi)^n}
    \int_X e^{-f}\frac{\omega^n}{n!}=\lim_{i\to\infty}\mathbb W_{X_i}(\xi_i).
\] Under our normalization, the shrinker
entropy is given by 
\[\boldsymbol\mu_X=\log W \mathbb -n\]  Thus
the pointed Nash-entropy inequality of Bamler
\cite[Proposition~5.2]{BamlerEntropy},
in the singular Ricci-shrinker-space form
\cite[Lemma~D.3 and Equation~(D.3)]{FangLi}, gives
\(\mathcal N_q(\tau)\geq\boldsymbol\mu_X\).
Every parabolic blow-up at \(q\) is its static Ricci-flat tangent cone
\cite{HZ}; continuity of pointed Nash entropy under noncollapsed
\(\mathbb F\)-convergence \cite[Theorem~2.10]{BamlerStructure}
 and the cone vertex heat kernel give
\[
 \log\Theta_{X,q}
 =\lim_{\tau\downarrow0}\mathcal N_q(\tau)
 \geq\boldsymbol\mu_X.
\]
Therefore \(\Theta_{X,q}\geq e^{-n}W\) and then
Theorem~\ref{thm:uniform-entropy} gives the desired inequality.
\end{proof}

\section{Volume growth and scalar-curvature mass}
\label{sec:effective-avr-gap}

We first recall a uniform estimate comparing the volume ratio at finite scales with the asymptotic volume ratio for shrinkers with maximal volume growth, and then use the Duistermaat--Heckman formula to estimate the volume growth and the scalar-curvature mass. A contradiction argument then yields a uniform lower bound for the volume ratio of shrinkers with maximal volume growth.

Applying the non-compact version of the Duistermaat--Heckman formula in \cite{PratoWu} to the soliton vector field directly and then using the relation between the volume growth and the integral of scalar curvature, we obtain the proof of Theorem \ref{int:polynomial-volume-scalar-mass}.

Combining K\"ahler reduction, the one-dimensionality of the affine base, and the multiplicity–degree formula, we prove Theorem~\ref{int:rank-one-volume-coefficient}.

\subsection{Uniform asymptotic volume ratios and volume ratio gap}

Throughout this section we use the complex normalization
\eqref{eq:canonical-shrinker-normalization}.  Put
\[
 \rho^2:=f+n=R_\omega+|\nabla^{1,0}f|^2
\]
and, for \(s>0\),
\begin{equation}\label{eq:potential-volume-scalar-mass}
 \mathsf V(s):=\int_{\{\rho^2<s\}}\frac{\omega^n}{n!},
 \qquad
 \mathsf S(s):=\int_{\{\rho^2<s\}}R_\omega\frac{\omega^n}{n!}.
\end{equation}
The coarea formula, together with
\(\Delta_\omega\rho^2=n-R_\omega\) and
\(|\nabla^{1,0}\rho^2|^2=\rho^2-R_\omega\), gives  \cite{CaoZhou}
\begin{equation}\label{eq:volume-scalar-mass-ode}
 s\mathsf V'(s)-n\mathsf V(s)
 =\mathsf S'(s)-\mathsf S(s).
\end{equation}
We use the normalized convention \cite{CLP}
\[
 \operatorname{AVR}(X,g)
 :=\lim_{r\to\infty}
 \frac{\operatorname{Vol}_g B_g(p,r)}{\omega_{2n}r^{2n}}.
\]

The following result holds more generally for Ricci shrinkers with a
uniform entropy lower bound. For simplicity of notation, we state it only
for K\"ahler--Ricci shrinkers. Its proof follows from the results in \cite{CaoZhou,CLP,LW,WW}.
\begin{proposition}
\label{prop:effective-avr}
Let \((X^n,\omega,J,f,p)\) be a  K\"ahler--Ricci
shrinker with 
\(p\in\operatorname*{argmin}_X f\). For any $\epsilon\in (0,1)$, there are \(r_0=r_0(n)\) and
\(C=C(n,\epsilon)\) such that, for
\(r\geq r_0\),
\begin{equation}\label{eq:effective-avr-estimate}
 0\leq \frac{\mathsf V(r^2/2)}{r^{2n}}
       -\omega_{2n}\operatorname{AVR}(X,g)
 \leq C r^{-1+\epsilon}.
\end{equation}
\end{proposition}

\begin{proof}By Theorem \ref{thm:uniform-entropy}, we have uniform entropy lower bound for K\"ahler--Ricci shrinkers. With exponent \(\varepsilon\in (0,1)\),
\cite[Corollary~6.24]{LWII} gives
\begin{equation}\label{eq:effective-avr-curvature-integral}
 \int_{B_g(p,r)}|\operatorname{Rm}_g|^{2-\epsilon}\,dV_g
 \leq C_{\epsilon} r^{2n+2\epsilon-2}\qquad(r\geq1).
\end{equation}
The centered uniform volume upper bound \cite{CaoZhou,HM} and H\"older's inequality
then imply
\begin{align}
 \int_{B_g(p,r)}R_\omega\,dV_g
 &\leq C_{\epsilon} r^{2n-1+\frac{\epsilon}{2-\epsilon}}.
 \label{eq:effective-avr-scalar-integral}
\end{align}
Nonnegatiity of the scalar curvature \cite{ChenScalar} and the estimate on soliton potential gives
\begin{equation}\label{eq:effective-avr-X-bound}
 0\leq \mathsf S(r^2/2)\leq C_\epsilon r^{2n-1+\epsilon}
 \qquad(r\geq r_0(n)).
\end{equation}

Let
\begin{equation}\label{eq:effective-avr-P}
 P(r)=\frac{\mathsf V(r^2/2)}{r^{2n}}
      -\frac{2\mathsf S(r^2/2)}{r^{2n+2}}.
\end{equation}
The computation in
\cite{CaoZhou,CLP,WW}
gives
\begin{equation}\label{eq:effective-avr-P-derivative}
 P'(r)=2\mathsf S(r^2/2)r^{-2n-1}
 \left(\frac{2n+2}{r^2}-1\right),
 \qquad
 \lim_{r\to\infty}P(r)
 =\omega_{2n}\operatorname{AVR}(X,g).
\end{equation}

For \(r\geq\sqrt{2n+2}\), the function \(P\) is nonincreasing.
Using \eqref{eq:effective-avr-X-bound} and integrating \(-P'\) gives
\[
 0\leq P(r)-\omega_{2n}\operatorname{AVR}(X,g)
 \leq C_\epsilon\int_r^\infty s^{-2+\epsilon}\,ds
 \leq C_\epsilon r^{-1+\epsilon}.
\]
By \(0\leq \mathsf S\leq n\mathsf V\), we have
\(
    2\mathsf S(r^2/2)r^{-2n-2}=O(r^{-2+\epsilon}).
\)
It then follows from \eqref{eq:effective-avr-P} that
\[
    0\leq \frac{\mathsf V(r^2/2)}{r^{2n}}
    -\omega_{2n}\operatorname{AVR}(X,g)
    \leq C_\epsilon r^{-1+\epsilon}.
\]
\end{proof}

 Let \((X,d,\omega,J,\rho^2)\in \overline{\mathcal KRS(n)}\), and let
\(\pi:X\to Y\) be its natural polarized Fano fibration.  By \cite{HZ}, there are a primitive circle generator \(\eta\), a proper
Hamiltonian \(u\) which is comparable to the square of the
distance function:
\begin{equation}\label{eq:singular-hamiltonian-potential-comparison}
 C^{-1}\rho^2-C\leq u\leq C\rho^2+C.
\end{equation} Moreover we know that there exists a number \(s_0\), a normal projective klt pair \((M,D)\),
a semiample class \(L_0\in\operatorname{Pic}(M)_{\mathbb Q}\) such that every reduction above \(s_0\) is \((M,D)\). Let the reduced ample class be
\begin{equation}\label{eq:singular-high-reduction-class}
 A_s=-(K_M+D)+sL_0,
 \qquad
 k:=\dim_{\mathbb C}Y=\operatorname{nd}(L_0)+1.
\end{equation}

\begin{theorem}
\label{thm:singular-high-level-dh}
We have
\begin{equation}\label{eq:singular-coarea-DH}
 \operatorname{Vol}_g(\{s_0<u<t\})
 =c_\eta\int_{s_0}^{t}A_s^{n-1}\,ds
 \qquad(t\geq s_0),
\end{equation}
and the integrand is a polynomial of degree
\(k-1\) with positive leading coefficient.  Consequently,
\begin{equation}\label{eq:singular-limit-growth-order}
 \operatorname{Vol}_g\{\rho^2<r^2/2\}\asymp r^{2k}.
\end{equation}
\end{theorem}

\begin{proof}
 Recall that we let $\mathcal{R}$ denote the regular locus of $X$ and its complement is denoted by $\mathcal{S}$.
 The codimension-four estimate gives
\(\operatorname{Vol}_g(\mathcal S)=0\), so the ordinary coarea formula on
\(\mathcal R\) gives, for \(t>s_0\) above all critical values,
\begin{equation}\label{eq:singular-regular-coarea}
 \operatorname{Vol}_g\{s_0<u<t\}
 =\int_{s_0}^{t}
   \left(\int_{u^{-1}(s)\cap\mathcal R}
   \frac{1}{|\nabla u|_g}\,dA_g\right)ds=\frac{1}{(n-1)!}\int_{s_0}^t
   \int_{Q_s}\omega_s^{n-1} ds.
\end{equation}
 where 
\( Q_s:=\bigl(X^s(\eta,s)\cap\mathcal R\bigr)/\mathbb C^*.
\)

Above the biggest critical value every orbit in the invariant singular locus has
complex dimension one.  By
\cite[Theorem~6.6 and its proof]{HZ},
\[
M\setminus Q_s
 =\bigl(X^s(\eta,s)\cap\mathcal S\bigr)/\mathbb C^*
\]
is analytic of complex codimension at least two.  \cite[Proof of Theorem 6.6]{HZ} shows that the reduced K\"ahler current
\(\omega_s\) on \(M\) has continuous, hence locally bounded, potentials.
By pulling back to its resolution, we know that its Bedford--Taylor top product
identifies its total mass with the intersection number:
\begin{equation}\label{eq:singular-BT-intersection-mass}
 \int_{Q_s}\omega_s^{n-1}
 =(2\pi)^{n-1}A_s^{n-1}.
\end{equation}
Combining \eqref{eq:singular-regular-coarea}--\eqref{eq:singular-BT-intersection-mass} proves
\eqref{eq:singular-coarea-DH}.
Put \(k=\operatorname{nd}(L_0)+1\).  Since \(L_0\) is semiample and
\(A_{s_0}\) is ample,
\[
 A_{s_0}^{n-k}L_0^{k-1}>0.
\]
 Thus
\(A_s^{n-1}\) has degree exactly \(k-1\) and positive leading coefficient.
After integration, \eqref{eq:singular-coarea-DH} gives
\(\operatorname{Vol}_g\{u<t\}\asymp t^k\).  The comparison
\eqref{eq:singular-hamiltonian-potential-comparison} between \(u\) and
\(\rho^2\) proves \eqref{eq:singular-limit-growth-order}.
\end{proof}

\begin{theorem}[=Theorem~\ref{int:uniform-avr-gap}]
\label{thm:uniform-avr-gap}
For every \(n\geq1\), there is \(\delta_n>0\) such that every $(X,g,f)\in \mathcal{KRS}(n,n)$ with \(p\in\operatorname*{argmin}_Xf\) satisfies
\[
 \lim_{r\to\infty}
 \frac{\operatorname{Vol}_gB_g(p,r)}{r^{2n}}\geq\delta_n.
\]
\end{theorem}

\begin{proof}
Suppose by contradiction that the claimed gap fails.
There are then K\"ahler--Ricci shrinkers \((X_i,g_i)\) for
which
\[
0< b_i:=\lim_{r\to\infty}
 \frac{\operatorname{Vol}_{g_i}B_{g_i}(p_i,r)}{r^{2n}}\rightarrow 0.
\]
   Theorem~\ref{thm:uniform-entropy}
gives a dimension-only entropy lower bound, and
Corollary~\ref{cor:weak-compactness} gives, after passing to a subsequence,
a centered limit \(X_\infty\).  Write
\(\pi_\infty:X_\infty\to Y_\infty\) for its natural polarized Fano
fibration.
By \cite[Proposition~7.1]{LZ} and  Theorem
\ref{thm:affine-base-dimension-constancy}, we have
\begin{equation}\label{eq:maximal-limit-base-dimension}
 \dim_{\mathbb C}Y_\infty=n.
\end{equation}
Theorem \ref{thm:singular-high-level-dh} gives object-dependent
\(c_\infty>0\)  such that
\begin{equation}\label{eq:limit-full-growth-lower}
 \operatorname{Vol}_{g_\infty}
 \{\rho_\infty^2<r^2/2\}
 \geq c_\infty r^{2n}\qquad(r\geq 1).
\end{equation}
However, volume convergence, together with the uniform estimate in
Proposition~\ref{prop:effective-avr}, shows that
\eqref{eq:limit-full-growth-lower} contradicts \(b_i\to 0\).
\end{proof}

\subsection{General volume growth and scalar-curvature mass}
\label{subsec:polynomial-dh-scalar-mass}

Let \((X^n,\omega,J,f)\) be a smooth non-compact K\"ahler--Ricci shrinker, let
\(p\) be a minimum point of \(f\), and let \(\pi:X\to Y\) be its natural
Fano fibration.  We use the same notation
\(\rho^2,\mathsf V,\mathsf S\) as above and put
\(k:=\dim_{\mathbb C}Y\geq 1\).

\begin{lemma}
\label{lem:eventual-dh-polynomial}
There are \(s_0>0\) and a polynomial
\[
 P_X(s)=\sum_{j=0}^k a_js^j,
 \qquad a_k>0,
\]
such that \(\mathsf V(s)=P_X(s)\) for every \(s\geq s_0\).
\end{lemma}

\begin{proof}
Let \(\mathbb T\) be the compact torus obtained as the closure
of the flow generated by \(\xi=J\nabla f\).  By \cite[Section~3]{SZ}, we can choose the moment-map
\(\mu:X\to\mathrm{Lie}(\mathbb T)^*\),
\begin{equation}\label{eq:soliton-moment-component}
 \rho^2=\langle\mu,\xi\rangle+n.
\end{equation}
The function \(\rho^2\) is proper and bounded below.  Moreover,
\(X^{\mathbb T}=Z(\xi)=\operatorname{Crit}\rho^2\) is compact by
\cite[Lemma~3.3 and Proposition~3.5]{SZ}.  It therefore has finitely many
connected components.

Since the \(\xi\)-flow is dense in
\(\mathbb T\), every nonzero normal weight pairs nontrivially with \(\xi\),
so \(\xi\) is regular in the sense of \cite[Definition 2.1]{PratoWu}.  The noncompact localization formula and its
fixed-component extension therefore apply by
\cite[Theorem~2.2 and the paragraph following Lemma~2.3]{PratoWu}. 
Let \(F\) be a connected component of \(X^{\mathbb T}\), let
\(\lambda_F:=\mu(F)\). 
The calculation in
\cite[Example~1.5 and proof of Theorem~3.2]{PratoWu} identifies  the contribution of
\(F\) to \((\rho^2)_*(\omega^n/n!)\) has the form
\[
 \mathbf 1_{[c_F,\infty)}(s)\,p_F(s-c_F)\,ds
\]
for a polynomial \(p_F\), where \(c_F=\langle\lambda_F,\xi\rangle=\rho^2|_F-n\) . 
Above the largest critical value their sum is one polynomial density, so
\(\mathsf V\) is polynomial there.  The same argument as in Theorem \ref{thm:singular-high-level-dh} gives
\(\mathsf V(s)\asymp s^k\); hence this polynomial has degree \(k\) and
positive leading coefficient.
\end{proof}

\begin{theorem}
\label{thm:volume-growth-limit}For $X\in \mathcal{KRS}(n,k)$, we have
\begin{equation}\label{eq:general-volume-growth-limit}
 0<\Theta_X:=\lim_{r\to\infty}
 \frac{\operatorname{Vol}_gB_g(p,r)}{r^{2k}}<\infty.
\end{equation}
\end{theorem}

\begin{proof}
The uniform estimate on the soliton potential \cite{CaoZhou,HM}
gives a constant
\(C\) such that, with
\[
 \Omega(r):=\{2\rho^2<r^2\},
\]
one has
\begin{equation}\label{eq:ball-potential-sandwich}
 \Omega(r-C)\subset B_g(p,r)\subset\Omega(r+C)
 \qquad(r\gg1).
\end{equation}
Lemma~\ref{lem:eventual-dh-polynomial} gives that
\(
 \operatorname{Vol}_g\Omega(r)=P_X(r^2/2)
\) is a polynomial, and 
therefore the conclusion
follows by squeezing. 
\end{proof}

\begin{theorem}
\label{thm:scalar-mass-polynomial}
There is a polynomial
\[
 Q_X(s)=\sum_{j=0}^k b_js^j
\]
such that \(\mathsf S(s)=Q_X(s)\) for every sufficiently large \(s\).  Its
coefficients are determined by
\begin{equation}\label{eq:scalar-mass-coefficient-recurrence}
 b_j=(n-j)a_j+(j+1)b_{j+1},
 \qquad 0\leq j\leq k,
 \qquad b_{k+1}:=0.
\end{equation}
 Moreover,
\begin{equation}\label{eq:scalar-mass-leading-asymptotic}
 \mathsf S(s)=(n-k)\mathsf V(s)+O_X(s^{k-1}),
\end{equation}
and hence
\begin{equation}\label{eq:ball-scalar-average-limit}
 \lim_{r\to\infty}
 \frac{\displaystyle
       \int_{B_g(p,r)}R_\omega\,\frac{\omega^n}{n!}}
      {\operatorname{Vol}_gB_g(p,r)}
 =n-k.
\end{equation}
\end{theorem}

\begin{proof}
Above the threshold in Lemma~\ref{lem:eventual-dh-polynomial}, the
right-hand side of \eqref{eq:volume-scalar-mass-ode} is therefore prescribed by a polynomial.  The solutions of
\eqref{eq:volume-scalar-mass-ode} then have the form
\[
 \mathsf S(s)=Q_X(s)+Ce^s,
\]
where \(Q_X\) is polynomial.  Since \(0\leq R_\omega\leq\rho^2<s\) on
\(\{\rho^2<s\}\),
\[
 0\leq\mathsf S(s)\leq s\mathsf V(s)=O_X(s^{k+1}).
\]
This forces \(C=0\).  Comparing the coefficient of \(s^j\) in
\eqref{eq:volume-scalar-mass-ode} gives
 \eqref{eq:scalar-mass-coefficient-recurrence}.  In particular,
\(b_k=(n-k)a_k\), proving
\eqref{eq:scalar-mass-leading-asymptotic}.

 Applying
\eqref{eq:ball-potential-sandwich} to both volume and scalar-curvature mass and using the nonnegativity of the scalar curvature, we obtain \eqref{eq:ball-scalar-average-limit} as a consequence of
  the polynomial asymptotics.
\end{proof}

\begin{corollary}\label{cor:maximal-scalar-mass-growth}
Suppose that \(k=n\).  Then the finite limit
\begin{equation}\label{eq:maximal-scalar-mass-coefficient}
 \Lambda_X:=\lim_{r\to\infty}r^{2-2n}
 \int_{B_g(p,r)}R_\omega\,\frac{\omega^n}{n!}
\end{equation}
exists. Moreover
if the shrinker is nonflat, then \(\Lambda_X>0\).
\end{corollary}

\begin{proof}
When \(k=n\), \eqref{eq:scalar-mass-coefficient-recurrence} gives
\(b_n=0\) and \(b_{n-1}=a_{n-1}\).  The 
\eqref{eq:ball-potential-sandwich} gives
\[
 \Lambda_X=2^{1-n}a_{n-1}.
\]
Assume now that \(X\) is nonflat.  By \cite[Theorem~1]{ChowLuYang},
\(
 R_\omega(x)d_g(p,x)^2\geq c_X
\)
outside a compact set and hence $\Lambda_X>0$.
\end{proof}

\begin{proof}[Proof of Theorem~\ref{int:polynomial-volume-scalar-mass}]
The three assertions follow respectively from
Lemma~\ref{lem:eventual-dh-polynomial},
Theorems~\ref{thm:volume-growth-limit} and
\ref{thm:scalar-mass-polynomial}, and
Corollary~\ref{cor:maximal-scalar-mass-growth}.
\end{proof}

\subsection{Uniform two-sided bounds in the minimal volume growth case}
We now turn to the proof of Theorem \ref{int:rank-one-volume-coefficient}. The following eigenvalue gap lemma is well known and follows from the Bochner formula and integration by parts; see
\cite[Proposition~3 and p.~203]{BakryEmery}. The polynomial growth estimates for \(h\) and \(\nabla h\) established in \cite[Lemma~4.2 and Corollary~4.3]{LZ} justify the integration by parts with respect to the weighted measure \(e^{-f}dV_g\).
\begin{lemma}
\label{lem:homogeneous-function-spectral-gap}
If a nonconstant homogeneous holomoprhic function \(h\) satisfies
\(\xi(h)=\sqrt{-1}\alpha h\), then \(\alpha\geq1\).
\end{lemma}

Let
\[
 B_{n-1}:=\sup\bigl\{(-K_Z)^{n-1}:
 Z\text{ is a smooth Fano }(n-1)\text{-fold}\bigr\},
 \qquad B_0:=1.
\]

\begin{theorem}[=Theorem \ref{int:rank-one-volume-coefficient}]
\label{thm:rank-one-volume-coefficient}
For $X\in \mathcal{KRS}(n,1)$, we have 
\begin{equation}\label{eq:uniform-rank-one-volume-coefficient}
 \frac{\pi(2\pi)^{n-1}}
 {(n-1)!\,nN_n^{n-1}B_{n-1}}
 \leq \Theta_X\leq
 \frac{\pi(2\pi)^{n-1}B_{n-1}}{(n-1)!}.
\end{equation}
\end{theorem}

\begin{proof}
Since \(Y\) is a normal affine cone of dimension one, \(Y\simeq\mathbb C\).
Choose a homogeneous coordinate \(h\) such that
\[
 \xi(h)=\sqrt{-1}\alpha h,\qquad \alpha>0,
\]
and let \(Z=\pi^{-1}(1)\) be a smooth general
fiber.

We first prove
\begin{equation}\label{eq:rank-one-coefficient-formula}
 \Theta_X
 =
 \frac{\pi}{\alpha}\operatorname{Vol}_\omega(Z)
 =
 \frac{\pi(2\pi)^{n-1}}{(n-1)!}
 \frac{(-K_Z)^{n-1}}{\alpha}.
\end{equation}
Take a rational approximation \(\eta\to\xi\), let \(T\) be the minimal
period of its flow, and write
\[
 \eta(h)=\sqrt{-1}\alpha_\eta h.
\]
Introduce the primitive \(2\pi\)-periodic generator
\[
 \widehat\eta:=\frac{T}{2\pi}\eta.
\]
Then
\[
 \widehat\eta(h)=\sqrt{-1}m_\eta h,
 \qquad
 m_\eta:=\frac{T\alpha_\eta}{2\pi}\in\mathbb Z_{>0}.
\]

Let \((M_\eta,D_\eta,L_\eta)\) be a regular K\"ahler reduction for this
primitive \(S^1\)-action, where \(D_\eta\) is the orbifold boundary and
\(L_\eta\) is the orbifold line bundle associated with the weight-one
representation of \(S^1\).  The reduced K\"ahler class satisfies
\[
 \frac{[\omega_\eta]}{2\pi}
 =
 c_1\!\left(-(K_{M_\eta}+D_\eta)+zL_\eta\right).
\]
On \(\pi^{-1}(\mathbb C^*)\), the function \(h\) is nowhere vanishing
and has \(S^1\)-weight \(m_\eta\).  Hence its restriction to the
corresponding level set gives a nowhere-vanishing section of
\(L_\eta^{\otimes m_\eta}\), up to dualizing \(L_\eta\).  Therefore
\[
 c_1(L_\eta)=0
 \quad\text{in }H^2(M_\eta,\mathbb Q),
\]
and consequently
\begin{equation}\label{eq:rank-one-reduced-class}
 [\omega_\eta]
 =
 2\pi c_1\bigl(-(K_{M_\eta}+D_\eta)\bigr).
\end{equation}
Identifying the K\"ahler reduction with the analytic Hilbert quotient, we know \(M_\eta\)
is the orbifold quotient of the slice \(Z\) by
\(\mu_{m_\eta}\).  Thus there is a finite quotient map
\[
 q_\eta:Z\longrightarrow M_\eta,
 \qquad
 \deg q_\eta=m_\eta.
\]
Then the ramification formula gives
\[
 K_Z=q_\eta^*(K_{M_\eta}+D_\eta).
\]
Hence
\[
 (-K_Z)^{n-1}
 =
 m_\eta
 \bigl(-(K_{M_\eta}+D_\eta)\bigr)^{n-1}.
\]
Together with \eqref{eq:rank-one-reduced-class} and adjunction formula, this yields
\begin{equation}\label{eq:rank-one-reduced-volume}
        \operatorname{Vol}_{\omega_\eta}(M_\eta)
 =
 \frac{1}{m_\eta}
 \frac{(2\pi)^{n-1}}{(n-1)!}(-K_Z)^{n-1}=\frac{1}{m_\eta}\operatorname{Vol}_\omega(Z).
\end{equation}
Then using coare formula as in \cite[Equation~(7.8)]{LZ} and \eqref{eq:rank-one-reduced-volume}, we obtain 
\[ \Theta_X
 =
 \lim_{\eta\to\xi}
 \frac{T_\eta}{2}
 \operatorname{Vol}_{\omega_{\eta}}(M_{\eta})= \lim_{\eta\to\xi} \frac{T_\eta}{2m_\eta}\operatorname{Vol}_\omega(Z)
 =
\lim_{\eta\to\xi} \frac{\pi}{\alpha_\eta}\operatorname{Vol}_\omega(Z)=\frac{\pi}{\alpha}\operatorname{Vol}_\omega(Z).
\]
which proves \eqref{eq:rank-one-coefficient-formula}.

It remains to bound \(\alpha\). Note that \(\pi\) is a flat morphism since the base is one dimensional.  Put
\(L=-N_nK_X\), which is \(\pi\)-very ample by
\eqref{eq:uniform-relative-embedding}.  At the fixed minimum
\(p\), let
\[
 E=\operatorname{div}(\pi^*h)
\]
be the central Cartier fiber and set
\(a_p=\operatorname{ord}_p(\pi^*h)\).  Then
\begin{align*}
 a_p
 &=\operatorname{mult}_pE
 \leq \deg_LE
 =\deg_LZ\\
 &=N_n^{n-1}(-K_Z)^{n-1}
 \leq N_n^{n-1}B_{n-1}.
\end{align*}
Here the multiplicity--degree inequality follows from
\cite[Proposition~5.1.9]{Laz1}, while equality of the fiber degrees
follows from flatness.

Choose holomorphic eigen-coordinates centered at \(p\) such that
\[
 \xi(z_j)=\sqrt{-1}\lambda_jz_j,\qquad \lambda_j\geq0.
\]
The weight identity at \(p\) gives
\[
 \sum_{j=1}^n\lambda_j=n-R_\omega(p)\leq n.
\]
If
\(\sum_{|I|=a_p}c_Iz^I
\)
is the first nonzero Taylor term of \(\pi^*h\), then every monomial
with \(c_I\neq0\) has \(\xi\)-weight \(\alpha\) since $h$ and $z_j$ are all homogeneous under the $\xi$-action.  Hence
\[
 \alpha
 =\sum_{j=1}^n I_j\lambda_j
 \leq a_p\sum_{j=1}^n\lambda_j
 \leq na_p
 \leq nN_n^{n-1}B_{n-1}.
\]
On the other hand,
Lemma~\ref{lem:homogeneous-function-spectral-gap} gives
\(\alpha\geq1\). 
Substituting these estimates into
\eqref{eq:rank-one-coefficient-formula} gives
\[
 \frac{\pi(2\pi)^{n-1}}
 {(n-1)!\,nN_n^{n-1}B_{n-1}}
 \leq\Theta_X\leq
 \frac{\pi(2\pi)^{n-1}B_{n-1}}{(n-1)!}.
\]
\end{proof}

\section{Uniform section estimates and global H\"ormander lifting}
\label{sec:gaussian-scalar-lifting}

This section develops the analytic tools used in the cone-isolation
argument.  We first prove uniform local estimates and weighted
integrability for homogenous pluri-anticanonical sections on every element of \(\overline{\mathcal{KRS}(n)}\).  We then
lift pluri-anticanonical sections from a limit space to the approximating
smooth shrinkers and construct the sequential peak sections needed in
Section~\ref{sec:limit-cone-setup}.  

\subsection{Uniform estimates}

Let
\[
 (X,d,\omega,J,f,p)\in\overline{\mathcal{KRS}(n)},
\quad
X=\mathcal R\sqcup  \mathcal S, \quad p\in\operatorname*{argmin}_X f.
\]
All tensorial quantities below are understood on \(\mathcal R\), while
metric balls are taken in \(X\).  Let \(h_{\omega^n}\) be the metric on
\(-K_{\mathcal R}\) induced by \(\omega^n\), and set
\[
 h=e^{-f}h_{\omega^n},
 \qquad
 \omega_{k+1}=(k+1)\omega \quad (k\geq0).
\]
The potential estimates \cite{CaoZhou,HM} and the shrinker equation passing to
the limit \cite{LLW} gives that for every \(D<\infty\), there is \(C_D=C(n,D)\) such
that
\begin{equation}\label{eq:uniform-potential-estimates}
 \frac12\bigl(d(p,x)-5\sqrt{2}n\bigr)_+^2
 \leq f(x)+n
 \leq \frac12\bigl(d(p,x)+\sqrt{2n}\bigr)^2,
 \quad
 R_\omega\leq C_D
 \quad\hbox{on }B(p,D)\cap\mathcal R.
\end{equation}

\begin{lemma}
\label{lem:fixed-power}
 For every \(D<\infty\), there are
\(C_D,c_D,r_D>0\), depending only on \(n\) and \(D\), such that the
following holds for every
\(X\in\overline{\mathcal{KRS}(n)}\).
For \(k\geq0\), \(x\in B(p,D)\), and \(0<r\leq r_D\), every
\[
 s\in H^0(\mathcal R,-kK_{\mathcal R})
\]
with finite \(L^2\)-norm on \(B(x,C_Dr)\cap\mathcal R\) satisfies
\begin{equation}\label{eq:scaled-section-estimate}
 \sup_{B(x,r)\cap\mathcal R}|s|_{h^k}^2
 \leq
 \frac{C_D(kr^2+1)}{((k+1)r^2)^n}
 \int_{B(x,C_Dr)\cap\mathcal R}|s|_{h^k}^2\,
       \omega_{k+1}^n.
\end{equation}
At \(\rho_k=c_D(k+1)^{-1/2}\), one also has
\[
 \sup_{B(x,\rho_k)\cap\mathcal R}
 \left(|s|_{h^k}^2+
 |\nabla s|_{\omega_{k+1}\otimes h^k}^2\right)
 \leq C_D\int_{B(x,C_D\rho_k)\cap\mathcal R}|s|_{h^k}^2
       \omega_{k+1}^n.
\]
Moreover, \(|s|_{h^k}\) extends uniquely to a locally Lipschitz function
on \(X\).
\end{lemma}

\begin{proof}
The value and gradient estimates follow from \cite[Lemma~3.10]{HZ}, which
is stated for singular K\"ahler--Ricci shrinkers and allows the center
\(x\) to lie anywhere in \(X\).  Its constants depend only on \(n,D\), and
an entropy lower bound, which is uniform here by
Theorem~\ref{thm:uniform-entropy}.  Taking \(r\) comparable to
\((k+1)^{-1/2}\) gives the natural-scale estimate.
Finally, \(|\nabla|s||\leq|\nabla s|\) on \(\mathcal R\), and the distance
induced by smooth curves in \(\mathcal R\) agrees with \(d\)
\cite[Theorem~1.1]{LLW}.  Hence the gradient estimate gives the asserted
locally Lipschitz extension.
\end{proof}

The next cutoff lemma is standard and follows from the Minkowski codimension 4 estimate and it allows us integrate by parts on the regular locus
without boundary terms along \(\mathcal S\).

\begin{lemma}[Codimension-four cutoffs]
\label{lem:codimension-four-cutoffs}
For every
\((X,d,\omega,J,f,p)\in\overline{\mathcal{KRS}(n)}\), there are smooth
functions \(\vartheta_\nu:\mathcal R\to[0,1]\) which vanish near
\(\mathcal S\), converge to \(1\) on compact subsets of \(\mathcal R\),
and satisfy
\begin{equation}\label{eq:LLW-cutoffs}
 \|\nabla\vartheta_\nu\|_{L^q(B(p,D)\cap\mathcal R)}
 \longrightarrow0
 \qquad\text{for every }D<\infty\text{ and }1\leq q<4.
\end{equation}
\end{lemma}

\begin{proof}
The Minkowski codimension estimate
\cite[Theorem~1.1(a) and the proof of Proposition~8.2]{LLW} gives, for
every \(D<\infty\) and \(0<\delta<1\), constants
\(C_{D,\delta}\) and \(s_{D,\delta}>0\) such that
\begin{equation}\label{eq:LLW-Minkowski}
 \Vol_g\!\left(
   \{d(\,\cdot\,,\mathcal S)<s\}\cap B(p,D)
 \right)
 \leq C_{D,\delta}s^{4-2\delta}
 \qquad(0<s<s_{D,\delta}).
\end{equation}
Fix \(q<4\) and choose \(0<\delta<(4-q)/2\).  A one-variable cutoff of
\(d(\,\cdot\,,\mathcal S)/s\) then satisfies
\[
 \int_{B(p,D)\cap\mathcal R}|\nabla\vartheta_s|^q\,\omega^n
 \leq C_{D,q}s^{-q}
       \Vol_g\!\left(
          \{d(\,\cdot\,,\mathcal S)<2s\}\cap B(p,D+1)
       \right)
 \leq C_{D,q}s^{4-q-2\delta}\longrightarrow0.
\]
Then the lemma follows by smoothing on \(\mathcal R\), diagonalize over integer radii and $q_j\nearrow 4$. 
\end{proof}

\subsection{Gaussian charges and weighted integrability}

We first record the Gaussian identities on a smooth
\((X,\omega,J,f,p)\in\mathcal{KRS}(n)\).  As always, \(H^0\)
denotes algebraic global sections.  Equip \(-K_X\) with
\(h=e^{-f}h_{\omega^n}\).  Then
\begin{equation}\label{eq:gaussian-polarization}
 \sqrt{-1}\Theta(h)
 =\operatorname{Ric}(\omega)+\sqrt{-1}\partial\bar\partial f
 =\omega.
\end{equation}
For every \(k\geq0\), set
\begin{equation}\label{eq:gaussian-inner-products}
 \langle s,t\rangle_{G,k}
 :=\int_X\langle s,t\rangle_{h^k}
       e^{-f}\frac{\omega^n}{n!},
 \qquad s,t\in H^0(X,-kK_X),
\end{equation}
where \(h^0=1\).  For \(k=0\), this is the Gaussian inner product on
holomorphic functions.

Let \(\mathbb T\) be the compact closure of the soliton flow and use its
canonical lift to anticanonical powers, as in \eqref{eq:canonical-lift}.  On
each algebraic character space, define the charge operator by
\begin{equation}\label{eq:gaussian-charge-operator}
 A:=-\sqrt{-1}\,\mathcal L_\xi,
 \qquad
 As=\langle\alpha,\xi\rangle s
 \quad\hbox{for }s\in H^0(X,-kK_X)_\alpha.
\end{equation}
The torus preserves \(f,\omega\), and \(h\).  Its character spaces are
therefore orthogonal for \eqref{eq:gaussian-inner-products}, and \(A\) is
self-adjoint on every
finite torus-invariant Gaussian subspace.

\begin{lemma}
\label{lem:homogeneous-section-weighted-integrability}
Let
\(X\in\overline{\mathcal{KRS}(n)}\),
\(k\in\mathbb Z_{\geq0}\), and
\[
 s\in H^0\!\left(X,\mathcal O_X(-kK_X)\right)
\]
be homogeneous for the canonical lift of the soliton torus.  Then
\begin{equation}\label{eq:limit-section-weighted-integrability}
 \int_{\mathcal R}|s|_{h^k}^2e^{-\tau f}\omega^n<\infty
 \qquad\text{for every }\tau>0.
\end{equation}
If \(s\ne0\) and \(As=\lambda s\), where
\(A=-\sqrt{-1}\mathcal L_\xi\), then
\begin{equation}\label{eq:limit-section-first-moment}
 \frac{\int_{\mathcal R}f|s|_{h^k}^2e^{-f}\omega^n}
      {\int_{\mathcal R}|s|_{h^k}^2e^{-f}\omega^n}
 =\frac{\lambda}{k+1}.
\end{equation}
\end{lemma}

\begin{proof}
By Lemma \ref{lem:fixed-power}, we know that
\(|s|_{h^k}\) is locally bounded.
 Note that 
 the canonical lift and the Chern connection are
related by
\begin{equation}\label{eq:charge-connection-operator}
 -\sqrt{-1}\,\mathcal L_\xi=\nabla^{h^k}_{\nabla^{1,0} f}+kf
 \qquad\hbox{on }H^0(X,-kK_X).
\end{equation}
Indeed, in a holomorphic coordinate frame
\(e=\partial_{z_1}\wedge\cdots\wedge\partial_{z_n}\) of \(-K_X\),
\[
 |e|_h^2=e^{-f}\det(g_{i\bar j}),\qquad
 \nabla^h_{\nabla^{1,0} f}e=\bigl(\nabla^{1,0} f(\log\det g)-|\nabla^{1,0} f|^2\bigr)e.
\]
Since $\xi=J\nabla f$ is a holomorphic vector field, 
\(
 \mathcal L_\xi e=-\sqrt{-1}\,\partial_j(g^{i\bar j}\partial_{\bar j}f)e.
\)
Consequently,
\[
 \begin{split}
 \bigl(-\sqrt{-1}\mathcal L_\xi-\nabla^h_{\nabla^{1,0} f}\bigr)e
 &=\bigl(-\Delta_\omega f+|\nabla^{1,0}f|^2\bigr)e
 =fe.
 \end{split}
\]
 Then \eqref{eq:charge-connection-operator} follows from tensoriality.

Assume \(s\ne0\) and write \(As=\lambda s\). On \(\mathcal R\), \eqref{eq:charge-connection-operator} gives
\begin{equation}\label{eq:limit-section-charge-divergence}
 \nabla^{h^k}_{\nabla^{1,0}f}s=(\lambda-kf)s,
 \qquad
 \operatorname{div}\!\left(
   e^{-f}|s|_{h^k}^2\nabla f\right)
 =2\bigl(\lambda-(k+1)f\bigr)|s|_{h^k}^2e^{-f}.
\end{equation}

For almost every regular value \(T\) of \(\rho^2\), apply the divergence
theorem on \(\{\rho^2<T\}\cap\mathcal R\) after inserting the cutoffs from
Lemma~\ref{lem:codimension-four-cutoffs}.  Local boundedness of \(|s|\)
and \(|\nabla f|\) makes the extra term at the singular set at most
\(C_T\|\nabla\vartheta_\nu\|_{L^1}\to0\).  Thus
\[
 0\leq e^{n-T}\int_{\{\rho^2=T\}\cap\mathcal R}
       |s|_{h^k}^2|\nabla f|\,dA
 =2\int_{\{\rho^2<T\}\cap\mathcal R}
   \bigl(\lambda-(k+1)f\bigr)|s|_{h^k}^2e^{-f}\omega^n.
\]
Choose \(B\) so that \((k+1)(B-n)-\lambda\geq1\).  Splitting the last
integral at \(\{\rho^2\leq B\}\) gives, for regular \(T>B\),
\[
 \int_{\{B<\rho^2<T\}}\bigl((k+1)f-\lambda\bigr)
       |s|_{h^k}^2e^{-f}\omega^n
 \leq
 \int_{\{\rho^2\leq B\}}
  |\lambda-(k+1)f|\,|s|_{h^k}^2e^{-f}\omega^n.
\]
The right side is finite and independent of \(T\).  Since the coefficient
on the left grows linearly in \(\rho^2\), the Gaussian norm and first
\(\rho^2\)-moment are finite.   This proves
\eqref{eq:limit-section-weighted-integrability} for \(\tau=1\).

Applying the same argument to the reflexive tensor power
\(s^{\otimes q}\), which has charge \(q\lambda\), gives
\[
 \int_{\mathcal R}|s|_{h^k}^{2q}e^{-f}\omega^n<\infty
\]
for every integer \(q\geq1\).  For \(0<\tau<1\), choose an integer
\(q\geq2\) with \(q\tau>1\).
H\"older's inequality, with
\(\sigma=(q\tau-1)/(q-1)>0\), yields
\[
 \int_{\mathcal R}|s|_{h^k}^2e^{-\tau f}\omega^n
 \leq
 \left(\int_{\mathcal R}|s|_{h^k}^{2q}e^{-f}\omega^n\right)^{1/q}
 \left(\int_{\mathcal R}e^{-\sigma f}\omega^n
       \right)^{(q-1)/q}<\infty.
\]
The last integral is finite by \eqref{eq:uniform-potential-estimates} and
the standard volume upper bound for shrinkers.
\end{proof}

\subsection{Global H\"ormander lifting and peak sections}

We now fix a centered convergence
\eqref{eq:general-LLW-convergence}.  The cutoff lemma lets us approximate
sections on \(\mathcal R\) by compactly supported smooth sections;
the H\"ormander estimate then corrects their \(\bar\partial\)-errors on
the smooth spaces \(X_i\).  Write
\[
 h_i=e^{-f_i}h_{\omega_i^n},
 \qquad
 h_\infty=e^{-f_\infty}h_{\omega_\infty^n}.
\]

Lemma~\ref{cor:ab-base-function-lifting} follows from the case $k=0$ of the following result.

\begin{proposition}
\label{prop:global-plurianticanonical-lifting}
Let \eqref{eq:general-LLW-convergence} be a centered
convergent sequence.  Fix \(k\geq0\), and let
\(s\in H^0(\mathcal R,-kK_{\mathcal R})\),
 satisfying
\begin{equation}\label{eq:limit-section-gaussian-integrability}
 \int_{\mathcal R}|s|_{h_\infty^k}^2e^{-f_\infty}
       \omega_\infty^n<\infty.
\end{equation}
There are sections
\[
 s_{i}\in H^0(X_i,-kK_{X_i}),
\]
which converge to \(s\) smoothly on
compact subsets of \(\mathcal R\).  Their norms converge uniformly
on every fixed pointed metric ball to the continuous extensions of
\(|s|_{h_\infty^k}\).  When \(k=0\), the functions \(s\) extend to
\(\mathcal O(X_\infty)\), and \(s_{i}\to s\) uniformly on every fixed
pointed ball.
\end{proposition}

\begin{proof}
The bound \eqref{eq:limit-section-gaussian-integrability} implies
local \(L^2\) control because \(f_\infty\) is locally bounded.
Lemma~\ref{lem:fixed-power} therefore extends
\(|s|_{h_\infty^k}\) locally Lipschitz to \(X_\infty\).  When \(k=0\),
the gradient estimate applies to \(s\) itself; the resulting continuous
function is holomorphic across the singular set by normality and the
Riemann extension theorem.

Let \(\chi_R\) be
a smooth function of \(f_\infty\), equal to one on
\(\{f_\infty\leq R\}\), equal to zero on
\(\{f_\infty\geq2R\}\), and satisfying \(|\chi_R'|\leq C/R\).  The estimate
\(|\nabla f_\infty|^2\leq f_\infty+n\) and
\eqref{eq:limit-section-gaussian-integrability} give
\begin{equation}\label{eq:outer-potential-cutoff-error}
 \|s\,\bar\partial\chi_R\|_{L^2(h_\infty^k e^{-f_\infty})}^2
 \leq \frac{C}{R}
  \int_{\{f_\infty\geq R\}}|s|_{h_\infty^k}^2e^{-f_\infty}
       \omega_\infty^n\longrightarrow0.
\end{equation}
For fixed \(R\), choose the cutoffs \(\vartheta_\nu\) from
Lemma~\ref{lem:codimension-four-cutoffs} on a centered ball containing
\(\{f_\infty\leq2R\}\).  The compact-sublevel bound and that lemma give
\begin{equation}\label{eq:general-singular-cutoff-error}
 \|\chi_Rs\,\bar\partial\vartheta_\nu
   \|_{L^2(h_\infty^k e^{-f_\infty})}\longrightarrow0.
\end{equation}
Thus \(v_{R,\nu}:=\chi_R\vartheta_\nu s\) is a compactly supported
smooth section of \(-kK_{\mathcal R}\), and its
\(\bar\partial\)-error tends to zero when first \(\nu\to\infty\) and then
\(R\to\infty\).

On each fixed support, the determinant of the complex-linear differential
of a compatible regular-locus embedding induces an identification of the
anticanonical powers.  In these gauges the bundle metrics and Chern
connections converge smoothly.  Transport \(v_{R,\nu}\) and extend it by
zero across its fixed collar to obtain
\[
 v_{i;R,\nu}\in C_c^\infty(X_i,-kK_{X_i}),
 \qquad
 \alpha_{i;R,\nu}:=\bar\partial v_{i;R,\nu}.
\]
Choose \(R_i,\nu_i\to\infty\) sufficiently slowly, and set
\(v_{i}:=v_{i;R_i,\nu_i}\) and
\(\alpha_{i,a}:=\bar\partial v_{i}\), and we can obtain
\begin{equation}\label{eq:global-section-cutoff-error}
 \int_{X_i}|\alpha_{i}|_{\omega_i\otimes h_i^k}^2e^{-f_i}
       \omega_i^n\longrightarrow0.
\end{equation}

To apply H\"ormander's estimate, we use the canonical identifications
\[
 -kK_{X_i}
 \simeq \Lambda^{n,0}\otimes(-(k+1)K_{X_i}),
 \qquad
 \Lambda^{0,1}\otimes(-kK_{X_i})
 \simeq \Lambda^{n,1}\otimes(-(k+1)K_{X_i}).
\]
Equip the form bundles with the Hermitian metrics induced by \(\omega_i\).
If \(\widetilde u\) and \(\widetilde\alpha\) denote the forms corresponding
to \(u\) and \(\alpha\), respectively, then
\[
 |\widetilde u|_{\omega_i\otimes h_i^{k+1}}^2
 =e^{-f_i}|u|_{h_i^k}^2, \quad  |\widetilde\alpha|_{\omega_i\otimes h_i^{k+1}}^2
 =e^{-f_i}|\alpha|_{\omega_i\otimes h_i^k}^2.
\]
Thus \cite[Theorem~5.1]{Demailly}, applied to the associated
\((n,1)\)-forms, gives smooth sections \(u_{i}\) of \(-kK_{X_i}\)
satisfying
\begin{equation}\label{eq:global-section-weighted-solution}
 \bar\partial u_{i}=\alpha_{i},\qquad
 \int_{X_i}|u_{i}|_{h_i^k}^2e^{-f_i}\omega_i^n
 \leq\frac1{k+1}\int_{X_i}
 |\alpha_{i}|_{\omega_i\otimes h_i^k}^2e^{-f_i}\omega_i^n
 \longrightarrow0.
\end{equation}
Then we let \(s_{i}=v_{i}-u_{i}\).
On each compact regular set, the transported sections converge smoothly to
\(s\), their \(\bar\partial\)-errors converge smoothly to zero, and
\eqref{eq:global-section-weighted-solution} gives local \(L^2\)-convergence
of \(u_{i}\) to zero.  Interior elliptic estimates therefore give smooth
convergence \(s_{i}\to s\) through the determinant gauges.

On every fixed centered ball, the transported sections have uniformly
bounded local \(L^2\)-norms, while
\eqref{eq:global-section-weighted-solution} controls the corrections.
Lemma~\ref{lem:fixed-power} gives equicontinuity of
\(|s_{i}|_{h_i^k}\).
Arzel\`a--Ascoli compactness and density of \(\mathcal R\) identify
every subsequential limit with the continuous extension of
\(|s|_{h_\infty^k}\).  When \(k=0\), the same argument applied to the
functions themselves gives pointed uniform convergence.
\end{proof}

Using Proposition~\ref{prop:global-plurianticanonical-lifting},
we may adapt the peak-section construction of Donaldson--Sun
\cite[Section~2.2]{DS}. This yields the following result; see also
\cite[Lemmas~4.1--4.2]{HZ}.
\begin{proposition}[Sequential Gaussian peak sections]
\label{prop:sequential-peaks}
For the convergence in \eqref{eq:general-LLW-convergence},
fix \(D\geq1\), \(x\in B(p_\infty,D)\),
\(x_i\to x\), and \(0<\epsilon<1\).
There is an integer \(q\geq\epsilon^{-1}\),  independently of
\(i\), and, for all sufficiently large \(i\), sections
\(
   s_{i,x}\in H^0(X_i,-qK_{X_i})
\)
such that
\begin{equation}
\label{eq:sequential-gaussian-peak}
 \sup_{y\in B(p_i,2D)}
 \left|
   |s_{i,x}(y)|_{h_i^q}^{2}
   -e^{-q d_i(x_i,y)^2/2}
 \right|<\epsilon
\end{equation}
and
\begin{equation}
\label{eq:sequential-peak-L2}
 \int_{B(p_i,4D)}|s_{i,x}|_{h_i^q}^{2}
       (q\omega_i)^n\leq C(n,D).
\end{equation}
In particular, there is an integer \(q_0\geq1\), and there are
constants \(a,c>0\), such that, for all large
\(i\), there are sections
\begin{equation}
\label{eq:vertex-sequential-peak}
 u_i\in H^0(X_i,-q_0K_{X_i}),
 \qquad
 \inf_{B(p_i,4a)}|u_i|_{h_i^{q_0}}\geq c.
\end{equation}
\end{proposition}

\section{Cone isolation and scalar curvature gap}
\label{sec:limit-cone-setup}

\subsection{Cone isolation}
Assume that the limit in \eqref{eq:general-LLW-convergence} is a
(singular) Ricci-flat K\"ahler metric cone
\(\mathcal C\), with vertex \(o=p_{\infty}\). Then in this case the associated Fano fibration of $\mathcal{C}$ is the identify map and $\mathcal{C}$ itself is a polarized affine cone with log terminal singularities \cite{HZ}.

\begin{proposition}
\label{prop:localized}
Assume that the limit in \eqref{eq:general-LLW-convergence} is a singular
Ricci--flat K\"ahler cone, and let
\(\pi_i:X_i\to(Y_i,o_i)\) be the natural Fano fibrations.  Then for \(E_i:=\pi_i^{-1}(o_i)\),
\[
          \sup_{x\in E_i}d_i(p_i,x)\longrightarrow0.
       \]
\end{proposition}

\begin{proof}
Since the limit is a metric cone and hence a polarized affine cone, then we can find finitely many holomoprhic function whose common zero set is the vertext. Then using Lemma \ref{cor:ab-base-function-lifting}, we can obtain finitely many holomorphic functions $g_{i,a}$ on $X_i$ satisfying the folowing for any $\epsilon>0$,
 for sufficiently large \(i\geq i(\epsilon)\),
\begin{equation}\label{eq:full-shell-separation}
 \sum_{a=1}^N|g_{i,a}(x)|^2>\delta(\epsilon)>0
 \qquad\text{whenever}\qquad
 0<\epsilon\leq d_i(p_i,x)\leq 2\epsilon,
\end{equation}and 
\begin{equation}\label{inside small}
        \sum_{a=1}^N|g_{i,a}(x)|^2\leq \frac{\delta(\epsilon)}{2}
 \qquad\text{whenever}\qquad
 d_i(p_i,x)\leq \epsilon_1\ll \epsilon.
\end{equation}
  The maximum principle makes every holomorphic function constant on
\(E_i\), since $E_i$ is connected.
The desired collapse of the central fiber then follows from
\eqref{eq:full-shell-separation} and \eqref{inside small}.
\end{proof}

\begin{theorem}[=Theorem~\ref{int:cone-isolation}]\label{thm:cone-isolation}
Let \((X_i^n,\omega_i,J_i,f_i,p_i)\) be a sequence of 
smooth K\"ahler--Ricci shrinkers centered at minima of
their soliton potentials
and assume that
\[
   (X_i,\omega_i,J_i, f_i,p_i)
   \longrightarrow
   (\mathcal C,\omega_{\mathcal C},J_{\mathcal C},f_{\infty},o)
\]
in the pointed \(\widehat C^\infty\)-Cheeger--Gromov sense,
where the target is a (possibly singular) Ricci-flat K\"ahler metric
cone.  Then
\(\mathcal C\) and $X_i$ for every sufficiently large
\(i\), are the Gaussian shrinker.
\end{theorem}

\begin{proof}
        Proposition~\ref{prop:sequential-peaks} gives an integer \(q\geq1\) and
constants \(a,c>0\) such that, for all sufficiently large \(i\), there are sections
\[
   \sigma_i\in H^0(X_i,-qK_{X_i})
\]
such that
\begin{equation}\label{eq:peak-nonvanishing}
   \inf_{B(p_i,2a)}|\sigma_i|_{h_i^q}\geq c.
\end{equation}
Proposition~\ref{prop:localized}(i) gives
\[
   E_i=\pi_i^{-1}(o_i)\subset B(p_i,a)
\]
for all sufficiently large \(i\).  Hence
the restriction \(\sigma_i|_{E_i}\) to the central fiber is nowhere zero. However, \(-K_{X_i}\) is
relatively ample. We obtain that the central fiber is zero-dimensional and hence \(\pi_i\) is an
isomorphism. Therefore by the standard result in commutative algebra, we obtain
\[
   (X_i,p_i)\simeq(Y_i,o_i)\simeq(\C^n,0).
\]
The uniqueness theorem \cite[Theorem~1.5]{LZ} therefore makes
\((X_i,\omega_i)\) Gaussian; see also
\cite[Corollary B]{conlon-deruelle}. 
\end{proof}

\subsection{Scalar curvature gap}

\begin{theorem}[=Theorem \ref{int:scalar-curvature-gap}]
\label{thm:entropy-scalar-gap}
For every \(n\geq1\), there is \(\epsilon_n>0\) such that a smooth shrinker
\((X^n,g,J,f,p)\) with 
\(p\in\operatorname*{argmin}_Xf\), is either Gaussian or satisfies
\begin{equation}
\label{eq:scalar-gap}
   \inf_{B_g(p,1)}R_\omega\geq\epsilon_n.
\end{equation}
\end{theorem}

\begin{proof}
Suppose the assertion fails.  Then there is a sequence of non-Gaussian
K\"ahler--Ricci shrinkers
\(
 (X_i,g_i,J_i,f_i,p_i)
\)
such that
\[
 \sigma_i:=\inf_{B_{g_i}(p_i,1)}R_{\omega_i}\longrightarrow0.
\]
By the uniform entropy bound established in Theorem \ref{int:uniform-entropy}, their shrinker entropies
are bounded uniformly from below.

Then we record the following consequence of the argument in \cite{LW} and 
\cite[Section~9]{LLW}: there exist constants \(C<\infty\) and
\(r_0>0\), depending only on the dimension and the entropy lower bound,
such that, whenever
\[
 0<\sigma:=\inf_{B_g(p,1)}R_\omega,
\]
there is a point \(z\in X\) satisfying
\begin{equation}\label{eq:regular-low-scalar-point}
 f(z)\leq C,\qquad
 R_\omega(z)\leq C\sigma,\qquad
 \operatorname{hr}_g(z)\geq r_0.
\end{equation}
Here \(\operatorname{hr}\) denotes the harmonic radius used in
\cite{LLW}.

Applying \eqref{eq:regular-low-scalar-point} to \(X_i\), choose
\(z_i\in X_i\) satisfying \eqref{eq:regular-low-scalar-point} for $\sigma_i$. Then
 after passing to a subsequence,
\[
 (X_i,g_i,f_i,p_i)
 \longrightarrow
 (X_\infty,g_\infty,f_\infty,p_\infty), \quad z_i\to
z_\infty \in \mathcal{R}
\]
in the pointed \(\widehat C^\infty\)-Cheeger--Gromov sense.   Then the smooth convergence gives
\[
 R_{\omega_\infty}(z_\infty)
 =
 \lim_{i\to\infty}R_{\omega_i}(z_i)
 =0.
\]
It follows from \cite[Proposition~9.9]{LLW} that
\((X_\infty,g_\infty,p_\infty)\) is a Ricci-flat metric cone with
vertex \(p_\infty\). Since $X_i$ are K\"ahelr, the limit is a Ricci-flat K\"ahler cone.
Finally,
Theorem~\ref{thm:cone-isolation} implies
that \(X_i\) is Gaussian for all sufficiently large \(i\), contradicting
the choice of the sequence.

\end{proof}

\section{Scalar decay and asymptotic conicality}
\label{sec:scalar-decay-ac}

In this section, we prove Theorem~\ref{int--thm:quadratic-scalar-decay-implies-ac}.
The main techniques are similar to those in \cite{LZ}, where methods from the
compact K\"ahler--Ricci flow were extended to the noncompact setting using
suitable barrier functions.

\subsection{Self-similar potential estimates}

Let $X$ be a non-compact smooth K\"ahler--Ricci shrinker. Put
\(
        \rho:=\sqrt{f+n}.
\)
Thus \(\rho^2=f+n\) is the standard nonnegative soliton potential,
\[
 R_\omega+|\nabla^{1,0}f|^2=\rho^2,
 \qquad
 \Delta_f\rho^2=n-\rho^2.
\]
Let \(\Phi_t\), \(t\in[-1,0)\), be the self-similar biholomorphisms generated
by \((2|t|)^{-1}\nabla f\), normalized by \(\Phi_{-1}=\operatorname{id}\), and
set
\begin{equation}\label{eq:ac-self-similar-flow}
        \omega_t=|t|\Phi_t^*\omega,
        \qquad
        F(x,t)=|t|\rho^2(\Phi_t(x)).
\end{equation}
Then \(\omega_t\) solves the K\"ahler--Ricci flow on \([-1,0)\), and
\[
        \partial_tF(x,t)=-R_\omega(\Phi_t(x)).
\]
Since complete ancient Ricci flows have nonnegative scalar curvature
\cite[Corollary~2.5]{ChenScalar}, \(F(x,t)\) is nonincreasing as
\(t\nearrow0\).

We let
\[
        \hat\omega_t=\omega-(t+1)\operatorname{Ric}(\omega),
        \qquad
        \varphi_t(x)=
        \int_{-1}^t
        \log\frac{\omega_s^n}{\omega^n}(x)\,ds .
\]
Then
\[
        \omega_t=\hat\omega_t+\sqrt{-1}\partial\bar\partial\varphi_t,
        \qquad
        \dot\varphi_t=\log\frac{\omega_t^n}{\omega^n},
        \qquad
        \partial_t\dot\varphi_t=-R_{\omega_t}.
\]

\begin{lemma}
\label{lem:ac-curvature-free-potential-estimate}
 One has
\begin{align}
 \dot\varphi_t(x)
 &=-\int_{-1}^t\frac{R_\omega(\Phi_s(x))}{|s|}\,ds,
 \label{eq:ac-potential-exact-integral}\\
 \varphi_t(x)
 &=-\int_{-1}^t(t-s)\frac{R_\omega(\Phi_s(x))}{|s|}\,ds.
 \label{eq:ac-potential-exact-integral-phi}
\end{align}
In particular,
\begin{equation}\label{eq:ac-curvature-free-potential-bound}
        -\rho^2(x)\le\varphi_t(x)\le0,
        \qquad
        0\le t\dot\varphi_t(x)\le \rho^2(x).
\end{equation}
Thus \(\varphi_t\) and \(t\dot\varphi_t\) are uniformly bounded on every
compact subset of \(X\), independently of \(t\).
\end{lemma}

\begin{proof}
Since \(R_{\omega_t}(x)=|t|^{-1}R_\omega(\Phi_t(x))\) and
\(\dot\varphi_{-1}=0\), integration of
\(\partial_t\dot\varphi_t=-R_{\omega_t}\) gives
\eqref{eq:ac-potential-exact-integral}.  A second integration gives
\eqref{eq:ac-potential-exact-integral-phi}.  Scalar curvature is nonnegative,
so \(\dot\varphi_t\le0\) and \(\varphi_t\le0\).

Write \(t=-\tau\), where \(0<\tau\le1\).  Since
\[
    \frac{d}{du}F(x,-u)=R_\omega(\Phi_{-u}(x)),
        \qquad F(x,-1)=\rho^2(x),
\]
the integral formulas and \(a/u\le1\), \((u-a)/u\le1\), give
\begin{align*}
 t\dot\varphi_t(x)
 &=\tau\int_\tau^1\frac{R_\omega(\Phi_{-u}(x))}{u}\,du
 \le \int_\tau^1R_\omega(\Phi_{-u}(x))\,du,\\
 -\varphi_t(x)
 &=\int_\tau^1\frac{u-\tau}{u}R_\omega(\Phi_{-u}(x))\,du
 \le \int_\tau^1R_\omega(\Phi_{-u}(x))\,du.
\end{align*}
The last integral equals \(\rho^2(x)-F(x,t)\le \rho^2(x)\), proving
\eqref{eq:ac-curvature-free-potential-bound}.
\end{proof}

\begin{lemma}[Convergence of the K\"ahler currents]
\label{lem:ac-current-convergence}
There is a locally bounded function \(\varphi_0\)
such that \(\varphi_t\downarrow\varphi_0\) pointwise and in
\(L^1_{\rm loc}\), and
\[
        \omega_t\longrightarrow
        \omega_0=\sqrt{-1}\partial\bar\partial(f+\varphi_0)
\]
in the sense of currents as \(t\nearrow0\).
Moreover, \(F_0:=\lim_{t\nearrow0}F(\cdot,t)\) exists and
\(F_0=\rho^2+\varphi_0\).
\end{lemma}

\begin{proof}
Lemma \ref{lem:ac-curvature-free-potential-estimate} gives both the
monotonicity of \(\varphi_t\) and the time-independent local bound
\(-\rho^2\le\varphi_t\le0\).  Hence \(\varphi_t\) has a locally bounded pointwise
limit \(\varphi_0\), and dominated convergence on compact sets gives
convergence in \(L^1_{\rm loc}\).  Passing to the limit against compactly supported test
forms in
\[
 \omega_t=\omega-(t+1)\operatorname{Ric}(\omega)
       +\sqrt{-1}\partial\bar\partial\varphi_t
\]
and using \(\omega-\operatorname{Ric}(\omega)
=\sqrt{-1}\partial\bar\partial f\) proves the assertion.
Since \(\partial_tF=-R_\omega(\Phi_t(\cdot))\le0\) and \(F\ge0\), the
pointwise limit \(F_0\) exists.  Letting \(t\nearrow0\) in
\eqref{eq:ac-potential-exact-integral-phi} gives
\[
        \varphi_0(x)
        =-\int_{-1}^0R_\omega(\Phi_s(x))\,ds,
        \qquad
        F_0(x)=\rho^2(x)+\varphi_0(x).
\]
\end{proof}

\subsection{Estimates along projective factor}

Fix the homogeneous ambient system furnished by the polarized
Fano fibration \cite{SZ}.  Thus
\[
 \pi=(h_1,\ldots,h_N):X\longrightarrow Y\subset\mathbb C^N,
 \qquad
 \Psi=[s_0:\cdots:s_M]:X\longrightarrow\mathbb P^M,
\]
where the \(h_\alpha\) are homogeneous affine generators and the
basepoint-free eigensections \(s_j\) belong to
\(H^0(X,(-K_X)^\ell)\) for one \(\ell\geq1\).  We choose the system so
that \((\pi,\Psi)\) is an embedding; hence the form
\(\omega_{\mathcal A}\) below is K\"ahler.  Let
\[
 \omega_{\rm aff}:=\sqrt{-1}\sum_{\alpha=1}^N
       \partial h_\alpha\wedge\bar\partial\overline{h_\alpha},
 \qquad
 \omega_{\mathcal A}:=\omega_{\rm aff}
       +\frac1\ell\Psi^*\omega_{\rm FS}.
\]
  Let \(h_\omega\) denote
the Hermitian metric on \(-K_X\) induced by \(\omega^n\), and set
\[
 S:=\sum_{j=0}^M|s_j|_{h_\omega^\ell}^2,
 \qquad
 \psi:=f-\frac1\ell\log S.
\]
Then we have 
\begin{equation}\label{equation for S}
        \omega
 =\frac1\ell\Psi^*\omega_{\rm FS}
  +\sqrt{-1}\partial\bar\partial\psi\qquad  \operatorname{Ric}(\omega)
 =\frac1\ell\Psi^*\omega_{\rm FS}
  +\sqrt{-1}\partial\bar\partial(-\frac1\ell\log S). 
\end{equation}

\begin{lemma}[Logarithmic control of the projective denominator]
\label{lem:projective-denominator-curvature-free}
If \(R_\omega\) is bounded, then there exists $C>0$ such that
\[
       |\log S|\leq C\log(1+\rho^2)
\]
\end{lemma}

\begin{proof}
Let \(X_0=\frac12\nabla f\), let \(\gamma_s\) be its flow, and choose a
regular value \(r_0>\sup_XR_\omega+1\) above all critical values of
\(\rho^2\).  Every point of \(\{\rho^2\geq r_0\}\) has the form
\(x=\gamma_s(y)\), with \(y\in\Sigma:=\{\rho^2=r_0\}\).  Along this flow,
\[
 \frac d{ds}\rho^2(\gamma_s y)=\rho^2-R_\omega.
\]
Consequently, if \(K:=\sup_XR_\omega\),
\begin{equation}\label{return time}
         \log\frac{\rho^2(x)}{r_0}
 \leq s
 \leq \log\frac{\rho^2(x)-K}{r_0-K}
 =\log\frac{\rho^2(x)}{r_0}+O(1).
\end{equation}
If \(\nu_j\) is the weight of \(s_j\) under the canonical lift of
\(X_0\), then
\[
 X_0\log S
 =\ell(n-R_\omega)
  +2\frac{\sum_j\nu_j|s_j|_{h_\omega^\ell}^2}{S}.
\]
The final quotient is bounded between the smallest and largest weights,
so \( |X_0\log S|\leq C\).  Since \(S\) has a positive minimum and a
finite maximum on the compact level \(\Sigma\), integration along
\(\gamma_s\) gives \( |\log S(x)|\leq C(1+s)\).  The return-time bounds \eqref{return time}
prove the assertion.
\end{proof}

Bounded scalar curvature also gives a polynomial lower bound for the
shrinker metric along the projective factor.
\begin{lemma}
\label{lem:projective-trace-polynomial}
Assume that \(R_\omega\) is bounded above.  Then there are
\(C,A<\infty\) such that
\begin{equation}\label{eq:projective-trace-polynomial}
\operatorname{tr}_\omega\Psi^*\omega_{\rm FS}
 \le C(1+\rho^2)^A.
\end{equation}
\end{lemma}

\begin{proof}
Put
\[
        U=\frac1\ell\operatorname{tr}_\omega\Psi^*\omega_{\rm FS}.
\]
For some \(C_0<\infty\), the holomorphic map Chern--Lu inequality gives,
first on \(\{U>0\}\),
\[
 \Delta_\omega\log U
 \ge
 \frac1{\ell U}
 \left\langle\operatorname{Ric}(\omega),\Psi^*\omega_{\rm FS}
 \right\rangle_\omega
 -C_0U-C_0.
\]
The flow of \(X_0=\frac12\nabla f\) induces a holomorphic vector field on
\(\mathbb P^M\), which we again denote by \(X_0\).  Equivariance gives
\[
 \mathcal L_{X_0}(\Psi^*\omega_{\rm FS})
 =
 \Psi^*(\mathcal L_{X_0}\omega_{\rm FS}).
\]
Since the target is compact and the projective system is fixed,
\[
 -C\Psi^*\omega_{\rm FS}
 \le
 \mathcal L_{X_0}(\Psi^*\omega_{\rm FS})
 \le
 C\Psi^*\omega_{\rm FS}.
\]
On the other hand,
\(\mathcal L_{X_0}\omega=\sqrt{-1}\partial\bar\partial f
=\omega-\operatorname{Ric}(\omega)\).  Differentiating
\(U=\ell^{-1}\operatorname{tr}_\omega\Psi^*\omega_{\rm FS}\) therefore
yields
\[
 X_0U
 =\frac1\ell\operatorname{tr}_\omega\!
   \left(\mathcal L_{X_0}(\Psi^*\omega_{\rm FS})\right)
  -U+
  \frac1\ell\left\langle
  \operatorname{Ric}(\omega),\Psi^*\omega_{\rm FS}
  \right\rangle_\omega.
\]
It follows directly that
\begin{equation}\label{eq:projective-drift-chern-lu}
        \Delta_f\log(U+1)\ge-C_0U-C_0.
\end{equation}

As noted above, we have
\[
 X_0\log S
 =\ell(n-R_\omega)
  +2\frac{\sum_j\nu_j|s_j|_{h_\omega^\ell}^2}{S}
 =\ell(n-R_\omega)+O(1).
\]  
By \eqref{equation for S}, we know that \(\Delta_\omega(-\frac1\ell\log S)=R_\omega-U\). Therefore we obtain
\begin{equation}\label{equation of log S}
             \Delta_f(-\frac1\ell\log S)=n-U+O(1).
\end{equation}

 Introduce the test function
\[
        G=(U+1)e^{\frac a\ell\log S}(1+\rho^2)^{-B},
\]where $B\gg a\gg C_0$.
Combining \eqref{eq:projective-drift-chern-lu}, \eqref{equation of log S} and
\(
 \Delta_f\log(1+\rho^2)
 =\frac{n-\rho^2}{1+\rho^2}
  -\frac{|\partial \rho^2|_\omega^2}{(1+\rho^2)^2},
\)
which is uniformly bounded, yields, after increasing \(C\),
\[
        \Delta_f\log G\ge(a-C_0)U-C.
\]
  Since
\(\Delta_fG=G(\Delta_f\log G+|\partial\log G|_\omega^2)\), it follows that
\begin{equation}\label{eq:projective-barrier-inequality}
        \Delta_fG\ge cG^2-CG.
\end{equation}

We use the standard shrinker cutoffs constructed as in
\cite[Section~3]{LW}.  Namely, fix one nonincreasing function
\(\eta\in C^\infty([0,\infty),[0,1])\) such that
\[
 \eta=1\ \text{on }[0,1],\qquad
 \eta=0\ \text{on }[2,\infty),\qquad
 (\eta')^2\le C\eta,\qquad |\eta''|\le C.
\]
For \(R\ge2n+1\), write
\[
        \eta_R=\eta(\rho^2/R).
\]
Since \(\rho\) is proper, \(\eta_R\) has compact support.  On its
transition region, \(R\le \rho^2\le2R\),
\[
 \frac{|\nabla\eta_R|^2}{\eta_R}\le\frac CR,
 \qquad
 \Delta_f\eta_R
 =\frac{\eta'}R(n-\rho^2)
  +\frac{\eta''}{R^2}|\partial \rho^2|_\omega^2
 \ge-\frac CR.
\]
Here the first term in \(\Delta_f\eta_R\) is nonnegative, since
\(\eta'\le0\) and \(\rho^2>n\), while
\(|\partial \rho^2|_\omega^2\le \rho^2\le2R\).

The compactly supported function \(\eta_RG\) attains a positive maximum at
some point \(x_R\) where \(\eta_R>0\).  There
\[
        \nabla G=-\frac{G}{\eta_R}\nabla\eta_R,
        \qquad
        \Delta_f(\eta_RG)\le0.
\]
Using \eqref{eq:projective-barrier-inequality} and these cutoff bounds, we
obtain at \(x_R\)
\[
 0
 \ge \Delta_f(\eta_RG)\ge c\eta_RG^2-C\eta_RG
      -\frac CRG
      -2G\frac{|\partial\eta_R|_\omega^2}{\eta_R}.
\]
After division by \(G(x_R)>0\), this gives
\(\eta_R(x_R)G(x_R)\le C\), uniformly in \(R\).  Hence \(G\le C\) on
\(\{\rho^2\le R\}\), and exhaustion gives \(G\le C\) globally.  Therefore
\[
        U+1\le Ce^{-\frac{a}{\ell}\log S}(1+\rho^2)^B
        \le C(1+\rho^2)^A.
\]
\end{proof}

\subsection{Metric control and curvature decay}After the preparation in the previous sections, the following argument is similar to that in \cite[Section~4]{LZ}.

\begin{lemma}
\label{lem:ac-potential-scalar-control}
Assume
\(
        R_\omega(x)\rightarrow 0,
\) as $x\rightarrow \infty$. Then there exists \(K\Subset X\) is compact such that  \(\inf_{X\setminus K} F_0\geq 1\) and for any $\delta>0$, there exists $C_\delta$ such that on $X\setminus K$, we have
\[
        |\dot \varphi_t|\leq \delta \log |t|^{-1}+C_\delta.
\]
\end{lemma}

\begin{proof}The first conclusion follows from the properness of the soliton potential and the boundedness of scalar curvature. Note that we have \(F_0(x)\ge 1\), then monotonicity
gives \(F(x,t)\ge 1\) for \(-1\le t<0\), and hence
\(\rho^2(\Phi_{|t|}(x))\ge 1/|t|\rightarrow \infty\) as $t\rightarrow 0$.  Then \eqref{lem:ac-curvature-free-potential-estimate} gives the desired estimate.
\end{proof}

\begin{proposition}
\label{prop:ac-projective-schwarz-lower-bound}
Assume \(
        R_\omega(x)\rightarrow 0,
\) as $x\rightarrow \infty$. 
let \(K\Subset X\) be compact and \(\inf_K F_0>0\). For every $\delta>0$, there is
\(C_{K,\delta}<\infty\) such that, for every \(-1\le t<0\),
\[
        C_{K,\delta}^{-1}|t|^{\delta}\omega_{\mathcal A}
        \le\omega_t\le
        C_{K,\delta}|t|^{-\delta}\omega_{\mathcal A}
        \qquad\text{on }K.
\]
\end{proposition}

\begin{proof}
Set
\[
 U=\frac1\ell\operatorname{tr}_{\omega_t}\Psi^*\omega_{\rm FS},
 \qquad H_t=(t+1)\dot\varphi_t-\varphi_t.
\]
The parabolic Chern--Lu inequality for the holomorphic map
\(\Psi:(X,\omega_t)\to(\mathbb P^M,\ell^{-1}\omega_{\rm FS})\) gives
\[
        (\partial_t-\Delta_{\omega_t})\log(U+1)
        \le C_0U+C_0 .
\]
The potential equation and \eqref{equation for S} give
\[
        (\partial_t-\Delta_{\omega_t})H_t
        =
        n-U-\Delta_{\omega_t}\psi .
\]
Since \((\partial_t-\Delta_{\omega_t})F=-n\), choosing \(a>C_0\) and
setting
\[
 Q=\log(U+1)+aH_t-a\psi-F-\bigl((a+1)n+C_0\bigr)t
\]
gives
\[
 (\partial_t-\Delta_{\omega_t})Q\le-(a-C_0)U\le0.
\]
Fix \(t_0<0\).  Equivariance,
Lemma~\ref{lem:projective-trace-polynomial}, and
\(\rho^2\circ\Phi_t=F/|t|\) yield
\[
 U(x,t)\le C_{t_0}(1+F(x,t))^A,
 \qquad -1\le t\le t_0,
\]
for some \(A<\infty\).
Since \(R_\omega\) is bounded,
\(\partial_tF=-R_\omega\circ\Phi_t\), and \(F(\cdot,-1)=\rho^2\),
\begin{equation}\label{eq:ac-uniform-properness}
        \rho^2(x)-\sup_XR_\omega\le F(x,t)\le\rho^2(x)
        \qquad(-1\le t<0).
\end{equation}
Lemma~\ref{lem:ac-potential-scalar-control} and 
Lemma~\ref{lem:projective-denominator-curvature-free} give
\[
Q(x,t)\to-\infty
\]
 uniformly for
\(-1\le t\le t_0\), and that \(\sup_XQ(\cdot,-1)<\infty\).  Applying the
maximum principle on \(\{\rho^2\le R\}\times[-1,t_0]\) and letting
\(R\to\infty\) gives
\[
        \sup_{X\times[-1,t_0]}Q\le\sup_XQ(\cdot,-1).
\]
Since the right side is independent of \(t_0\), \(Q\) is bounded above on
\(X\times[-1,0)\).

On the compact set \(K\), \(F\) and \(\psi\) are bounded and then
Lemma~\ref{lem:ac-potential-scalar-control} gives for every $\delta>0$, there exists $C_{K,\delta}$ such that for all \(t\in[-1,0),\) 
\[
        \frac1\ell\operatorname{tr}_{\omega_t}\Psi^*\omega_{\rm FS}
        \le C_{K,\delta} |t|^{-\delta}
        \qquad\text{on }K.
\]
 \cite[Proposition~4.4]{LZ} gives
\(
        \omega_t\ge C_K^{-1}\omega_{\rm aff}
\) on $K$ for $-1\le t<0.$
Consequently, \[\omega_t\ge C_{K,\delta}^{-1}|t|^\delta\omega_{\mathcal A}.\]  Finally,
\(\dot\varphi_t\le0\) by
\eqref{eq:ac-potential-exact-integral}, so
\[
        \omega_t^n=e^{\dot\varphi_t}\omega^n\le\omega^n
        \le C_K\omega_{\mathcal A}^n
        \qquad\text{on }K.
\]
The determinant upper bound together with the lower eigenvalue bound gives upper bound of $\omega_t$, completing the proof.
\end{proof}

\begin{proof}[Proopf of Theorem \ref{int--thm:quadratic-scalar-decay-implies-ac}]
Choose \(r_0>\sup_XR_\omega+1\) and above all critical values and set
\(\Sigma=\{\rho^2=r_0\}\).
  Then we can choose
a compact neighborhood \(K\) of \(\Sigma\) with \(\inf_K F_0>0\).
Proposition~\ref{prop:ac-projective-schwarz-lower-bound} and the interior
estimates \cite[Theorem~1.1--(ii)]{SWein} then give
\(C_\delta>0\) such that
\[
 \sup_{\Sigma\times[-1   ,0)}
 |\operatorname{Rm}(\omega_t)|_{\omega_t}\le C_\delta |t|^{-\delta}.
\]
Then, arguing as in the proof of \cite[Theorem~4.6]{LZ}, we obtain that the Riemannian curvature for the shrinker metric tends to zero at infinity. It then follows from
\cite{KW,MW} that \((X,\omega)\) is asymptotically conical.
\end{proof}

\section{Conjectures}
\label{sec:filtered-partial-c0-flat}

We collect two conjectures suggested by the results in this paper.  

\begin{conjecture}[Uniform two-sided volume-ratio bound]
\label{conj:uniform-volume-growth}
For every \(n\geq1\) and \(k\in\{1,\ldots,n\}\), there is a constant
\(C_{n,k}>1\) such that every \(X\in\mathcal{KRS}(n,k)\) satisfies
\[
 C_{n,k}^{-1}
 \leq
 \lim_{r\to\infty}\frac{\operatorname{Vol}_g B_g(p,r)}{r^{2k}}
 \leq
 C_{n,k}.
\]
\end{conjecture}

Theorems~\ref{int:uniform-avr-gap} and
\ref{int:rank-one-volume-coefficient} prove the conjecture for \(k=n\) and
\(k=1\), respectively.  For \(2\leq k\leq n-1\), the current
argument gives only bounds depending on the specific K\"ahler--Ricci shrinker.




\

\begin{conjecture}[Flat convergence]
\label{conj:flat-convergence}
For every convergent sequence as in \eqref{metric-convergence}, after passing
to a subsequence, there are a scheme \(B\),
points \(b_i\to b_\infty\) in its analytic topology, flat morphisms
\[
 \mathcal X\longrightarrow B,
 \qquad
 \mathcal Y\longrightarrow B,
\]
and a projective \(B\)-morphism \(\mathcal X\to\mathcal Y\) such that
\[
 (\mathcal X_{b_i}\to\mathcal Y_{b_i})\simeq(X_i\to Y_i),
 \qquad
 (\mathcal X_{b_\infty}\to\mathcal Y_{b_\infty})
 \simeq(X_\infty\to Y_\infty).
\]

\end{conjecture}

The main missing ingredient in Conjecture~\ref{conj:flat-convergence}
is an appropriate algebraic boundedness theorem for K\"ahler--Ricci shrinkers, closely related to the boundedness conjectures formulated in
\cite[Conjecture~6.7]{SZ} and \cite[Conjecture~3.4]{Odaka}.

As a preliminary step towards Conjecture~\ref{conj:flat-convergence}, we have the following uniform estimate on the dimension of homogeneous pluri-anticanonical sections.
\begin{proposition}[Uniform finiteness for bounded-charge ]
\label{lem:bounded-charge-finiteness}
For every \(m\in\mathbb Z_{\geq0}\), and \(D\in\mathbb R\),
there is \(N=N(n,m,D)<\infty\) such that every
\(X\in\mathcal{KRS}(n)\) satisfies
\begin{equation}\label{eq:bounded-charge-space}
 \mathcal H_{m,\leq D}(X)
 :=\bigoplus_{\langle\alpha,\xi\rangle\leq D}
 H^0(X,-mK_X)_\alpha,
 \qquad \dim\mathcal H_{m,\leq D}(X)\leq N.
\end{equation}
For a sequence we write
\(\mathcal H_{i,m,\leq D}:=\mathcal H_{m,\leq D}(X_i)\).
\end{proposition}

\begin{proof}
Suppose that the asserted uniform bound fails.  Then we could choose shrinkers with
\(\dim\mathcal H_{i,m,\leq D}\to\infty\).  Choose
\(p_i\in\operatorname*{argmin}_{X_i}f_i\).  By
Corollary~\ref{cor:weak-compactness}, pass to a limit as in
\eqref{eq:general-LLW-convergence}.  By 
\eqref{eq:uniform-potential-estimates}, we know that if
\(f_i+n\leq A\), then
\[
 d_i(p_i,\,\cdot\,)\leq5\sqrt2\,n+\sqrt{2A}.
\]

Pass to a further subsequence and relabel so that
\(\dim\mathcal H_{i,m,\leq D}\geq i\). Distinct character spaces are
orthogonal, and the self-adjoint charge operator admits a
Gaussian-orthonormal eigenbasis of the finite-dimensional space
\(\mathcal H_{i,m,\leq D}\).  Choose and order its first \(i\) vectors as
\[
 s_{i,1},\ldots,s_{i,i}\in\mathcal H_{i,m,\leq D},
 \qquad A_i s_{i,a}=\lambda_{i,a}s_{i,a},
 \qquad \lambda_{i,a}\leq D.
\]
For every unit vector in these ordered lists,
\eqref{eq:limit-section-first-moment} gives
\[
 \int_{X_i}(f_i+n)|s_{i,a}|_{h_i^m}^2e^{-f_i}
       \frac{\omega_i^n}{n!}
 =n+\frac{\lambda_{i,a}}{m+1}.
\]
Since \(f_i+n\geq0\), it follows that
\(
 -(m+1)n\leq\lambda_{i,a}\leq D.
\)
and
\begin{equation}\label{eq:lemma-bounded-charge-tail}
 \int_{\{f_i+n>A\}}|s_{i,a}|_{h_i^m}^2e^{-f_i}
       \frac{\omega_i^n}{n!}
 \leq\frac{n+|D|/(m+1)}{A}.
\end{equation}

For each fixed \(A\), equation
\eqref{eq:uniform-potential-estimates} places \(\{f_i+n\leq A\}\) in a
uniform ball \(B(p_i,R_A)\).  The upper bound for \(f_i\) on the slightly
larger ball and Gaussian unit normalization give the local unweighted bound
\[
 \int_{B(p_i,R_A+2)}|s_{i,a}|_{h_i^m}^2\omega_{m+1,i}^n
 \leq n!(m+1)^n
 e^{\sup_{B(p_i,R_A+2)}f_i}\|s_{i,a}\|_{G,i,m}^2
 \leq C_{A,m}.
\]
Consequently, Lemma~\ref{lem:fixed-power} and
elliptic bootstrappin on the regular-locus give, after further subsequences,
smooth convergence of every \(s_{i,a}\) on compact subsets of
\(\mathcal R\) to a holomorphic
section \(s_{\infty,a}\). 
By the uniform integral estimate and the singular set has measure zero, we know that the limit holomoprhic section $s_{i,a}$ has charge at most $D$. Then the diagonal construction gives an infinite orthonormal family
\(\{s_{\infty,a}\}\) with charge at most $D$, which gives a contradiction.
\end{proof}

\bibliographystyle{alpha}
\bibliography{ref.bib}
\end{document}